\documentclass[12pt,a4paper]{article}

\usepackage{amsmath}

\usepackage{amssymb,amsthm}
\usepackage[frenchb,english]{babel}

\usepackage[T1]{fontenc}
\usepackage[latin1]{inputenc}
\usepackage{graphics}
\usepackage{graphicx}
\usepackage{epsfig}
\usepackage{latexsym}
\usepackage{amscd}
\usepackage{color}
\usepackage{shadow}
\usepackage{a4wide}

\newcommand\R{{\mathbb R}}
\newcommand\N{{\mathbb N}}

\def\cdotv{\raise 2pt\hbox{,}}

\newcommand{\Be}[3]{\dot{B}^{{#1},{#3}}_{#2}}
\newcommand{\BL}[4]{\dot{B}^{{#1},{#3}}_{#2}({L}^{#4}_t)}
\newcommand{\BLT}[4]{\dot{B}^{{#1},{#3}}_{#2}({L}^{#4}_T)}

 \def\pasdegrille{\let\grille =
\pasgrille}  \def\aat#1#2#3{ \divide
\dimen1 by 48 \dimen3=\dimen1 \multiply \dimen1 by #1 \advance \dimen1
by -\dimen3 \divide \dimen1 by 101 \multiply \dimen1 by 100 \divide
\dimen2 by \count11 \multiply \dimen2 by #2
\setbox0=\hbox{#3}\ht0=0pt\dp0=0pt \rlap{\kern\dimen1 \vbox
to0pt{\kern-\dimen2\box0\vss}}\dimen1= \wd1 \dimen2=\ht1}
\def\pasgrille{ \count12= \dimen1 \divide \count12 by 50 \divide
\dimen2 by \count12 \count11 =\dimen2 \ \divide \dimen1 by 48
\setlength{\unitlength}{\dimen1} \smash{\rlap{\ }} \dimen1= \wd1
\dimen2=\ht1 } \def\grille{ \count12= \dimen1 \divide \count12 by 50
\divide \dimen2 by \count12 \count11 =\dimen2 \ \divide \dimen1 by 48
\setlength{\unitlength}{\dimen1} \smash{\rlap{\graphpaper[1](0,0)(50,
\count11)}} \dimen1= \wd1 \dimen2=\ht1 }

\edef\@tempa#1#2{\def#1{\mathaccent\string"\noexpand\accentclass@#2 }}
\@tempa\ring{017}

\theoremstyle{plain}
\newtheorem{thm}{Theorem}[section]
\newtheorem{prop}{Proposition}[section]

\newtheorem{lemma}{Lemma}[section]
\numberwithin{equation}{section}

\newtheorem{dfn}{Definition}[section]
\newtheorem{rmq}{Remark}[section]

\begin{document}

\bibliographystyle{plain}

\title{On the energy critical Schr\"{o}dinger equation in $3D$ non-trapping domains}
\author{Oana Ivanovici  \footnote{The first author was partially supported by grant A.N.R.-07-BLAN-0250}
 \\Universite Paris-Sud, Orsay,\\
Mathematiques, Bat. 430, 91405 Orsay Cedex, France\\
oana.ivanovici@math.u-psud.fr\\
\and
Fabrice Planchon  \footnote{The second author was partially supported by A.N.R. grant ONDE NON
   LIN}\\Laboratoire Analyse, G\'eom\'etrie
   \& Applications, UMR 7539 du CNRS,\\ Institut Galil\'ee, Universit\'e Paris
   13,\\ 99 avenue J.B. Cl\'ement, F-93430 Villetaneuse\\fab@math.univ-paris13.fr}

\date{}
\maketitle

\maketitle \baselineskip 17pt
\begin{abstract}
We prove that the quintic Schr\"{o}dinger equation with Dirichlet
boundary conditions is locally well posed for $H_0^{1}(\Omega)$ data on any
smooth, non-trapping domain $\Omega\subset\mathbb{R}^{3}$. The key
ingredient is a smoothing effect in $L^5_x(L^2_t)$ for the linear
equation. We also derive
scattering results for the whole range of defocusing subquintic Schr\"odinger
equations outside a star-shaped domain.
\end{abstract}


\section{Introduction}
The Cauchy problem for the semilinear Schr\"odinger equation in $\R^3$
is by now relatively well-understood: after seminal results by
Ginibre-Velo \cite{give85} in the energy class for energy subcritical
equations, the issue of local well-posedness in the critical Sobolev
spaces ($\dot H^{\frac 3 2-\frac 2 {p-1}})$ was settled in
\cite{cawe90}. Scattering for large time was proved in \cite{give85}
for energy subcritical defocusing equations, while the energy critical
(quintic) defocusing equation was only recently successfully tackled
in \cite{ckstt06}. The local well-posedness relies on Strichartz
estimates, while scattering results combine these local results with
suitable non concentration arguments based on Morawetz type estimates. On
domains, the same set of problems remains an elusive target, due to
the difficulty in obtaining Strichartz estimates in such a setting. In
\cite{bgt03}, the authors proved Strichartz estimates with an
half-derivative loss on non trapping domains: the non trapping
assumption is crucial in order to rely on the local smoothing
estimates. However, the loss resulted in well-posedness results for
strictly less than cubic nonlinearities; this was later improved to
cubic nonlinearities in \cite{ramona} (combining local smoothing and
semiclassical Strichartz near the boundary) and in \cite{OanaPS} (on
the exterior of a ball, through precised smoothing effects near the
boundary). Recently there were two significant improvements, following
different strategies:
\begin{itemize}
\item in \cite{plve08}, Luis Vega and the second author obtain an $L^4_{t,x}$ Strichartz
  estimate which is scale invariant. However, one barely misses
  $L^4_t(L^\infty(\Omega))$ control for $H^1_0$ data, and therefore
  local wellposedness in the energy space was improved to all subcritical (less than
  quintic) nonlinearities, but combining this Strichartz estimate with
  local smoothing close to the boundary and the full set of Strichartz
  estimates in $\R^3$ away from it. Scattering was also obtained for
  the cubic defocusing equation, but the lack of a good local
  wellposedness theory at the scale invariant level ($\dot H^\frac 1
  2$) led to a rather intricate incremental
  argument, from scattering in $\dot H^\frac 1 4$ to scattering in $H^1_0$;
\item in \cite{OanaSS}, the first author proved the full set of
  Strichartz estimates (except for the endpoint) outside stricly
  convex obstacles, by following the strategy pioneered in
  \cite{smso95} for the wave equation, and relying on the
  Melrose-Taylor parametrix. In the case of the Schr\"odinger
  equation, one obtains Strichartz estimates on a semiclassical time
  scale (taking advantage of a \og finite speed of propagation\fg{}
  principle at this scale), and then upgrading to large time results from combining them with the
  smoothing effect (see \cite{bu02} for a nice presentation of such an
  argument, already implicit in \cite{stta02}). Therefore, one obtains
  the exact same local wellposedness theory as in the $\R^3$ case,
  including the quintic nonlinearity, and scattering holds for all
  subquintic defocusing nonlinearities, taking advantage of the a
  priori estimates from \cite{plve08}.
\end{itemize}

In the present work, we aim at providing a local wellposedness theory
for the quintic nonlinearity outside non trapping obstacles, a case
which is not covered by \cite{OanaSS}. From explicit computations with
gallery modes (\cite{OanaCE}), one knows that the full set of optimal
Strichartz estimates does not hold for the Schr\"odinger equation on a
domain whose boundary has at least one geodesically convex point;
while this does not preclude a scale invariant Strichartz estimate
with a loss (like the $L^4_t(L^\infty_x)$ estimate in $\R^3$ which is
enough to solve the quintic NLS), it suggests to bypass the issue and
use a different set of estimates, which we call smoothing estimates:
in $\R^3$, these estimates may be stated as follows,
\begin{equation}
  \label{eq:ST}
  \|\exp(it\Delta) f\|_{L^4_x(L^2_t)}\lesssim \|f\|_{\dot H^{-\frac 1
    4}},
\end{equation}
from which one can infer various estimates by using Sobolev in time and/or
in space. Formally,  \eqref{eq:ST} is an immediate consequence of the
Stein-Tomas restriction theorem in $\R^3$ (or, more accurately, its
dual version, on the extension): let $\tau>0$ be a fixed radius, one sees $\hat f(\xi)$
as a function on $|\xi|=\sqrt \tau$, and applies the extension
estimate, with $\delta$ the Dirac function and $\mathcal{F}$ the space
Fourier transform
$$
\| \mathcal{F}^{-1}(\delta(\tau-|\xi|^2) \hat f(\xi))\|_{L^4_x}
\lesssim \|\hat f(\xi)\|_{L^2(|\xi|=\sqrt\tau)}.
$$
Summing over $\tau$ yields the $L^2$ norm of $f$ on the RHS, while on
the left we use Plancherel in time and Minkowski to get
\eqref{eq:ST}. A similar estimate holds for the wave equation,
replacing $\sqrt \tau=|\xi|$ by $\tau=\pm|\xi|$, and usually goes
under the denomination of square function (in time) estimates. In a
compact setting (e.g. compact manifolds) a substitute for the Stein-Tomas
theorem is provided by $L^p$ eigenfunction estimates, or better yet,
spectral cluster estimates. In the context of a compact manifold with
boundaries, such spectral cluster estimates were recently obtained by
Smith and Sogge in \cite{smso06}, and provided a key tool for solving
the critical wave equation on domains, see \cite{bulepl07,bupl07}. In
this paper, we apply the same strategy to the Schr\"odinger equation:
\begin{itemize}
\item we derive an $L^5(\Omega;L^2_I)$ \og smoothing\fg{} estimate for
  spectrally localized data on compact manifolds with boundaries,   from the spectral cluster $L^5(\Omega)$
  estimate; here $I$ is a time   interval whose size is such that
  $|I||\sqrt{-\Delta_D}| \sim 1$;
\item we decompose the solution to the linear Schr\"odinger equation
  on a non trapping domain into two main regions: close to the
  boundary, where we can view the region as embedded into a $3D$
  punctured torus, to which the previous semi-classical
  estimate may be applied, and then sumed up using the local smoothing
  effect; and far away from the boundary where the
  $\R^3$ estimates hold.
\item Finally, we patch together all estimates to obtain an estimate
  which is valid on the whole exterior domain. Local wellposedness in
  the critical Sobolev space $\dot H^{\frac 3 2-\frac 2 {p-1}}$
  immediatly follows for $3+2/5<p\leq 5$, and together with the a
  priori estimates from \cite{plve08}, this implies scattering for the
  defocusing equation for $3+2/5<p<5$. The remaining range $3\leq
  p\leq 3+2/5$ is sufficiently close to $3$ that, as alluded to in \cite{plve08}, a suitable
  modification of the arguments from \cite{plve08} yields scattering
  as well.
\end{itemize}
\begin{rmq}
  Clearly, such smoothing estimates are better suited to \og
  large\fg{} values of $p$: the restriction $3+2/5<p$ for the critical
  wellposedness is directly linked to the exponent $5$ in the spectral
  cluster estimates; in $\R^3$, where the correct (and optimal !)
  exponent is $4$, one may solve down to $p=3$ by this method, while
  the Strichartz estimates allow to solve at scaling level all the way
  to the $L^2$ critical value $p=1+4/3 $. 
\end{rmq}

\section{Statement of results}
Let $\Theta$ be a compact, non-trapping obstacle in $\mathbb{R}^{3}$
and set $\Omega=\mathbb{R}^{3}\setminus\Theta$. By $\Delta_{D}$ we
denote the Laplace operator with constants coefficients on $\Omega$. For $s\in \R$, $p,q\in
[1,\infty]$ we denote by $\Be s p q (\Omega)=\Be s p q$ the
Besov spaces on $\Omega$, where the spectral localization in their
definition is meant to be with respect to $\Delta_D$. We write
$L^{p}_x=L^p(\Omega)$ and $\dot H^{\sigma}=\Be s 2 2$ for the
Lebesgue and Sobolev spaces on $\Omega$. It will be useful to
introduce the Banach-valued Besov spaces $\BL s p q r$, and we refer
to the Appendix for their definition. Whenever $L^p_t$ is replaced by
$L^p_T$, it is meant that the time integration is restricted to the
interval $(-T,T)$.

We aim at studying wellposedness for the energy critical equation on
$\Omega\times\mathbb{R}$, with Dirichlet boundary condition,
\begin{equation}\label{eqquint}
i\partial_{t}u+\Delta_{D}u=\pm |u|^{4}u,\quad u_{|\partial\Omega}=0,\quad u_{|t=0}=u_{0}
\end{equation}
and more generally
\begin{equation}\label{eqp}
i\partial_{t}u+\Delta_{D}u=\pm |u|^{p-1}u,\quad u_{|\partial\Omega}=0,\quad u_{|t=0}=u_{0}
\end{equation}
with $p<5$.
\begin{thm}\label{thmleq}(Well-posedness for the quintic Schr\"{o}dinger equation)
Let $u_{0}\in H^{1}_0(\Omega)$. There exists $T(u_0)$ such that the
quintic nonlinear equation \eqref{eqquint} admits  a unique solution $u\in
C([-T,T],H^{1}_0(\Omega))\cap \BLT 1 5 2 {\frac
  {20}{11}}$.
Moreover, the solution is global in time and scatters in $H^1_0$ if the data is small. 
\end{thm}
The previous  theorem extends to the following subcritical range:
\begin{thm}\label{thmleqsub}
Let $3+\frac 2 5< p<5$, $s_p=\frac 3 2-\frac 2 {p-1}$ and  $u_{0}\in
\dot H^{s_p}$. There exists $T(u_0)$ such that the nonlinear equation
\eqref{eqp} admits a unique solution $u\in
C([-T,T],\dot H^{s_p})\cap \BLT {s_p} 5 2 {\frac
  {20}{11}}$. Moreover the
solution is global in time and scatters in $\dot H^{s_p}$ if the data is small. 
\end{thm}
\begin{rmq}
  We elected to state both theorems for Dirichlet boundary conditions
  mostly for sake of simplicity. Indeed, both results hold with Neuman
  boundary conditions: the key ingredients for our linear estimates
  are known to hold for Neuman, see \cite{smso06,bgt03}, while the
  nonlinear mappings from our appendix rely on \cite{ip08} (where all
  relevant estimates can be proved to hold in the Neuman case). 
\end{rmq}
Finally, we consider the long time asymptotics for \eqref{eqp} in the
defocusing case, namely the $+$ sign on the left; in this situation,
we are indeed restricted to the Dirichlet boundary conditions, as we
rely on a priori estimates from \cite{plve08}.
\begin{thm}\label{thmleqscat}
Assume the domain $\Omega$ to be the exterior of a star-shaped compact
obstacle (which implies $\Omega$ is non trapping). Let $3\leq p<5$,
and  $u_{0}\in H^{1}_0(\Omega)$. There exists a
unique global in time solution $u$, which is in the energy class, $C(\R,H^{1}_0(\Omega))$, to the
nonlinear equation \eqref{eqp} in the defocusing case ($+$ sign in
\eqref{eqp}). Moreover, this solution scatters for large times:
there exists two scattering states $u^\pm \in H^1_0(\Omega)$ such that
$$
\lim_{t\leftarrow \pm \infty} \| u(x,t)-e^{it\Delta_D}
u^\pm\|_{H^1_0(\Omega)} =0.
$$
\end{thm}
As mentioned in the introduction, the (global) existence part was
dealt with in \cite{plve08}; for the scattering part, the $p=3$ case
was also dealt with in
\cite{plve08}. In the setting of Theorem \ref{thmleqsub}, one may
adapt the usual argument from the $\R^n$ case, combining a priori
estimates and a good Cauchy theory at the critical regularity;  this
provides a very short argument in the range $3+2/5< p<5$. In the
remaining range, namely $3<p\leq 3+2/5$,
one unfortunately needs to adapt the intricate proof from
\cite{plve08}, and this leads to a much lenghtier proof; we provide it
mostly for the sake of completeness. This type of argument may however
be of relevance in other contexts.
\section{Smoothing type estimates}\label{smoothing}
We start with definitions and notations. Let $\psi(\xi^2)\in
C^{\infty}_{0}(\mathbb{R}\setminus \{0\})$ and
$\psi_j(\xi^2)=\psi(2^{-2j}\xi^2)$. On the domain $\Omega$, one has
the spectral resolution of the Dirichlet Laplacian, and we may define
smooth spectral projections  $\Delta_j=\psi_j (-\Delta_D)$ as
continuous operators on $L^2$. Moreover, these operators are
continuous on $L^p$ for all $p$, and if $f$ is Hilbert-valued and such
that $\|\|f\|_H\|_{L^p(\Omega)}<+\infty$, then the operators
$\Delta_j$ are continuous as well on $L^p(H)$. We refer to \cite{ip08}
for an extensive discussion and references. We simply point out that
if $H=L^2_t$, then $\Delta_j$ is continuous on all $L^p_x L^q_t$ by
interpolation with the obvious $L^p_t(L^p_x)$ bound and duality.

In this section we concentrate on estimates for the linear
Schr\"{o}dinger equation on $\Omega\times\mathbb{R}$ with Dirichlet
boundary conditions,
\begin{equation}\label{LS}
i\partial_{t}u_L+\Delta_{D}u_L=0,\quad {u_L}_{|\partial\Omega}=0,\quad {u_L}_{|t=0}=u_{0}
\end{equation}
\begin{thm}\label{thmestim}
The following local smoothing estimate holds for the homogeneous
linear equation \eqref{LS},
\begin{equation}\label{l5l2bis}
\|\Delta_j u_L\|_{L^{5}_x L^{2}_t}\lesssim 2^{-\frac{j}{10}}\|\Delta_j u_{0}\|_{L^{2}_x}.
\end{equation}
Moreover, let $2\leq q\leq \infty$, then 
\begin{equation}\label{l5l2}
\|\Delta_j u_L\|_{L^{5}_x L^{q}_t}\lesssim 2^{-j(\frac{2}{q}-\frac{9}{10})}\|\Delta_j u_{0}\|_{L^{2}_x}.
\end{equation}
\end{thm}
Consider now the inhomogeneous equation,
\begin{equation}\label{ILS}
i\partial_{t}v+\Delta_{D}v=F,\quad {v}_{|\partial\Omega}=0,\quad {v}_{|t=0}={0}.
\end{equation}
From Theorem \ref{thmestim}, we will obtain the following set of estimates:
\begin{thm}\label{thmestimin}
Let $2\leq q<r\leq +\infty$, then
\begin{equation}\label{l5l2in}
\|\Delta_j v\|_{C_t(L^2_x)}+2^{j(\frac{2}{q}-\frac{9}{10})}\|\Delta_j
v\|_{L^5_xL^q_t}\lesssim
2^{-j(\frac{4}{r}-\frac{9}{5})}
\|\Delta_j F\|_{L^{\frac 5 4}_xL^{r'}_t},
\end{equation}
with $1/r+1/r'=1$.
\end{thm}
Combining the previous theorems with the results from \cite{plve08},
we finally state the set of estimates which will be used later for
\begin{equation}\label{TLS}
i\partial_{t}u+\Delta_{D}u=F_1+F_2,\quad {u}_{|\partial\Omega}=0,\quad {v}_{|t=0}={u_0}.
\end{equation}
\begin{thm}\label{thmestiminbis}
Let $2<r\leq +\infty$, then
\begin{multline}\label{l5l2tot}
\|\Delta_j u\|_{C_t(L^2_x)}+2^{\frac{j}{10}}\|\Delta_j
u\|_{L^5_xL^2_t}+2^{-\frac 3 4 j}\|\Delta_j u\|_{L^4_{t,x}}\lesssim
\|\Delta_j u_0\|_{L^2_x}+
2^{-j(\frac{4}{r}-\frac{9}{5})}
\|\Delta_j F_1\|_{L^{\frac 5 4}_xL^{r'}_t}\\
{}+2^{-\frac 1 4 j}\|\Delta_j F_2\|_{L^{\frac 4 3}_{t,x}},
\end{multline}
with $1/r+1/r'=1$.
\end{thm}

\subsection{Proof of Theorem \ref{thmestim}}
Let $\tilde{\psi}\in C^{\infty}_{0}(\mathbb{R}\setminus\{0\})$ be such
that $\tilde{\psi}=1$ on the support of $\psi$: hence, if $\tilde
\Delta_j$ denotes the corresponding localization operator,
$\tilde{\Delta}_j\Delta_j =\Delta_j$. We now split the solution
of the linear equation
$\Delta_j u_L={\tilde \Delta_j}\Delta_j u_{L}$
as a sum of two terms ${\tilde \Delta_j}\chi
\Delta_j u_{L}+{\tilde \Delta_j}(1-\chi)\Delta_j u_{L}$,
where $\chi\in C^{\infty}_{0}(\mathbb{R}^{3})$ is compactly supported
and it is equal to $1$ near the boundary $\partial\Omega$.

\subsubsection{\og Far\fg{} from the boundary:  ${\tilde
    \Delta_j}(1-\chi)\Delta_j  u_L$} \label{sectfarfrombound}
Set $w_{h}(t,x)=(1-\chi)\Delta_j e^{it\Delta_{D}}u_{0}(x)$. Then $w_{h}$ satisfies
\begin{equation}\label{estloin}
\left\{\begin{array}{ll}
i\partial_{t}w_{h}+\Delta_{D}w_{h}=-[\Delta_{D},\chi]\Delta_j u_L,\\
w_{h}|_{t=0}=(1-\chi)\Delta_j u_{0}.
\end{array}
\right.
\end{equation}
Since $\chi$ is equal to $1$ near the boundary $\partial\Omega$, we
can view the solution to \eqref{estloin} as the solution of a
problem in the whole space $\mathbb{R}^{3}$. Consequently, the Duhamel
formula writes
\begin{equation}\label{est}
w_{h}(t,x)=e^{it\Delta_0}(1-\chi)\Delta_j u_{0}-\int_{0}^{t} e^{i(t-s)\Delta_0}[\Delta_{D},\chi]\Delta_j u_L(s)ds,
\end{equation}
where $\Delta_0$ is the free Laplacian on $\mathbb{R}^{3}$ and therefore
the contribution of $e^{it\Delta_0}(1-\chi)\Delta_j u_{0}$
satisfies the usual Strichartz estimates. We have thus reduced the
problem to the study of the second term in the right hand-side of
\eqref{est}. Ideally, one would like to remove the time restriction
$s<t$ and use a variant of the Christ-Kiselev lemma. However, this
would miss the endpoint case $q=2$. Instead, we recall the following lemma:
\begin{lemma}[Staffilani-Tataru \cite{stta02}]
\label{ST}
  Let $x\in \R^n$, $n\geq 3$ and let $f(x,t)$ be compactly supported in space, such
  that $f\in L^2_t(H^{-\frac 1 2})$. Then the solution  $w$ to
  $(i\partial_t +\Delta_0) w=f$ with $w_{|t=0}=0$, is such that
  \begin{equation}
    \label{eq:stta02inh}
    \|w\|_{L^2_t(L^{\frac{2n}{n-2}}_x)} \lesssim \|f\|_{L^2_t(H^{-\frac 1 2})}.
  \end{equation}
\end{lemma}
In fact, one may shift regularity in \eqref{eq:stta02inh} without
difficulty. Now, the proof in \cite{stta02} relies on a decomposition into traveling
waves, to which homogeneous estimates are then applied. We can
therefore use the $L^4_x(L^2_t)$ smoothing estimate, Sobolev in space,
and extend the conclusion of Lemma \ref{ST} to
\begin{equation}
  \label{eq:extST}
      \|w\|_{L^5_x(L^{2}_t)} \lesssim \|f\|_{L^2_t(H^{-\frac 1 2-\frac 1
      {10}})},
\end{equation}
where we chose to conveniently shift the regularity to the right
handside.

We now take $f=-[\Delta_{D},\chi]\Delta_j u_{L}\in L^{2}_t H^{-1/2- 1/10}_{\text{comp}}(\Omega)$ and 
\[
\|[\Delta_{D},\chi]\Delta_j
u_{L}\|_{L^{2}H^{-1/2-1/10}_{\text{comp}}}\lesssim\|\Delta_j
u_{L}\|_{L^{2}\dot H^{1/2- 1/10}(\Omega)}\lesssim \|\Delta_j
u_{0}\|_{\dot H^{1/10}(\Omega)},
\]
from which the smoothing estimates follow
\begin{multline}\label{striloinbord}
\|(1-\chi){
  \Delta_j}u_{L}\|_{L^{5}(\mathbb{R}^{3})L^2_t}\lesssim
\|(1-\chi)\Delta_j u_{0}\|_{\dot H^{-\frac 1
    {10}}(\mathbb{R}^{3})}+\|[\Delta_{D},\chi]\Delta_j
u_{L}\|_{L^{2}H^{-1/2-1/10}_{\text{comp}}}\\
\lesssim \|\Delta_j u_{0}\|_{\dot H^{-\frac
    1 {10}}(\Omega)}.
\end{multline}
We conclude using the continuity properties of $\tilde{\Delta}_j$
which were recalled at the beginning of Section \ref{smoothing} (e.g. see \cite[Cor.2.5]{ip08}).
In fact, using \eqref{striloinbord}, we get
\begin{align*}
\|{\tilde \Delta_j}(1-\chi)\Delta_j u_{L}\|_{L^5_x L^2_t } & \lesssim 
\|(1-\chi)\Delta_j u_{L}\|_{L^5_x L^2_t }\\
 & \lesssim 2^{-\frac {j}{10}}\|\Delta_j u_{0}\|_{L^{2}(\Omega)},
\end{align*}
where we have used the spectral localization $\Delta_j $ to estimate 
\[
\|\Delta_j u_{0}\|_{\dot H^{\sigma}(\Omega)}\simeq 2^{\sigma j}\|\Delta_j u_{0}\|_{L^{2}(\Omega)}.
\]

\subsubsection{\og Close\fg{} to the boundary:  ${\tilde \Delta_j}\chi \Delta_j u_L$} 
For $l\in\mathbb{Z}$ let $\varphi_{l}\in
C^{\infty}_{0}(((l-1/2)\pi,(l+1)\pi))$ equal to $1$ on
$[l\pi,(l+1/2)\pi]$. We set $v_j ={\tilde \Delta_j}\chi \Delta_j u_L$
and for $l\in\mathbb{Z}$ we set $v_{j,l}=\varphi_{l}(2^j t)v_j $. 
We have 
\begin{align}\label{estvh}
\|v_j
\|^{2}_{L^{5}(\Omega)L^{2}(\mathbb{R})}  =
\|\sum_{l\in\mathbb{Z}}v_{j,l}\|^{2}_{L^5_x L^2_t }
  \simeq
 \|\|\sum_{l\in\mathbb{Z}}v_{j,l}\|^{2}_{L^2_t }\|_{L^{5/2}_x}\nonumber\\
  \lesssim
 \|\sum_{l\in\mathbb{Z}}\|v_{j,l}\|^{2}_{L^2_t }\|_{L^{5/2}_x}
   \leq\sum_{l\in\mathbb{Z}}\|v_{j,l}\|^{2}_{L^5_x L^2_t },
\end{align}
where for the first inequality we used the fact that the supports in
time of $\varphi_{l}$ are almost orthogonal. In order to estimate
$\|v_j \|^{2}_{L^5_x L^2_t }$ it will be thus
sufficient to estimate each
$\|v_{j,l}\|^{2}_{L^5_x L^2_t }$. The equation
satisfied by $\tilde{v}_{j,l}:=\varphi_{l}(2^j t)\chi\Delta_{j}u_{L}$ is
\begin{equation}\label{eqvh}
i\partial_{t}\tilde{v}_{j,l}+\Delta_{D}\tilde{v}_{j,l}=-(\varphi_{l}(2^{j}t)[\Delta_{D},\chi]\Delta_j u_{L}-i2^j \varphi'_{l}(2^{j}t)\chi \Delta_j  u_{L}),
\end{equation}
where we stress that $\tilde{v}_{j,l}$ vanishes outside the time interval $(2^{-j}(l-1/2)\pi,2^{-j}(l+1)\pi)$. We denote $V_{j,l}$ the
right hand side in \eqref{eqvh}, namely
\begin{equation}\label{grandvhl}
V_{j,l}:=-\varphi_{l}(2^{j}t)[\Delta_{D},\chi]\Delta_j u_{L}+i2^{j}\varphi'_{l}(2^{j}t)\chi \Delta_j  u_{L}.
\end{equation}

Let $Q\subset\mathbb{R}^{3}$ be an open cube sufficiently large such
that $\partial\Omega$ is contained in the interior of $Q$. We denote
by $S$ the punctured torus obtained from removing the obstacle $\Theta$ (recall that
$\Omega=\mathbb{R}^{3}\setminus\Theta$) in the compact manifold
obtained from $Q$ with periodic boundary conditions on $\partial
Q$. Notice that defined in this way $S$ coincides with the Sina\"i
billiard. Let also $\Delta_{S}:=\sum_{j=1}^{3}\partial^{2}_{j}$ denote
the Laplace operator on the compact domain $S$.

On $S$, we may define a spectral localization operator using
eigenvalues $\lambda_k$ and eigenvectors $e_k$ of $\Delta_S$: if $f=\sum_k c_k e_k$, then
\begin{equation}
  \label{eq:DeltaS}
  \Delta^S_j f=\psi(2^{-2j}\Delta_S) f=\sum_k \psi(2^{-2j}\lambda_k^2)
  c_k e_k.
\end{equation}
\begin{rmq}\label{rmqeqsd}
Notice that in a neighborhood of the boundary, the domains of $\Delta_{S}$ and $\Delta_{D}$ coincide, thus if $\tilde{\chi}\in C^{\infty}_{0}(\mathbb{R}^{3})$ is supported near $\partial\Omega$ then
\[
\Delta_{S}\tilde{\chi}=\Delta_{D}\tilde{\chi}.
\] 
In order to apply estimates on the manifold $S$, we will need to
relocalize close to the obstacle. Consider $\chi_{1}\in
C^{\infty}_{0}(\mathbb{R}^{3})$ supported near the boundary and equal
to $1$ on the support of $\tilde{\chi}$, we will write
\begin{equation}\label{lemsec11}
\chi_{1}{\tilde \Delta_j}\tilde{\chi}=\chi_{1}\tilde{\Delta}^{S}_j \tilde{\chi}+\chi_{1}(\tilde{\Delta}_j-\tilde{\Delta}^S_j) \tilde{\chi},
\end{equation}
with the expectation that the difference term is smoothing.
\end{rmq}

In what follows let $\tilde{\chi}\in C^{\infty}_{0}(\mathbb{R}^{3})$
be equal to $1$ on the support of $\chi$ and be supported in a
neighborhood of $\partial\Omega$ such that on its support the operator
$-\Delta_{D}$ coincide with  $-\Delta_{S}$. From their respective
definition, $\tilde{v}_{j,l}=\tilde{\chi}\tilde{v}_{j,l}$, $V_{j,l}=\tilde{\chi}V_{j,l}$,
consequently $\tilde{v}_{j,l}$ will also solve the following equation
on the compact manifold $S$
\begin{equation}\label{eqvhs}
\left\{\begin{array}{ll}
i\partial_{t}\tilde{v}_{j,l}+\Delta_{S}\tilde{v}_{j,l}=V_{j,l},\\
\tilde{v}_{j,l}|_{t<h(l-1/2)\pi}=0,\quad \tilde{v}_{j,l}|_{t>h(l+1)\pi}=0.
\end{array}
\right.
\end{equation}
Therefore we can write the Duhamel formula either for the last equation \eqref{eqvhs} on
$S$, or for the equation \eqref{eqvh} on $\Omega$. We now apply $\tilde{\Delta}_{j}$  and use that
$v_{j.l}=\tilde{\Delta}_{j}\tilde{v}_{j,l}$,
$\tilde{\chi}\tilde{v}_{j,l}=\tilde{v}_{j,l}$ and
$\tilde{\Delta}_{j}\tilde{\chi}=\chi_{1}\tilde{\Delta}^{S}_{j}\tilde{\chi}+(1-\chi_{1})\tilde{\Delta}_{j}\tilde{\chi}+\chi_1(\tilde{\Delta}_j-\tilde{\Delta}^S_j)\chi$, which 
yields
\begin{multline}\label{formevjld}
v_{j,l}(t,x)=\chi_{1}\int_{h(l-1/2)\pi}^{t}
e^{i(t-s)\Delta_{S}}\tilde{\Delta}^S_j
V_{j,l}(s,x)ds\\
{}+(1-\chi_{1})\int_{h(l-1/2)\pi}^{t}
e^{i(t-s)\Delta_{D}} \tilde{\Delta}_{j}V_{j,l}(s,x)ds\\
{}+\chi_1(\tilde{\Delta}_j-\tilde{\Delta}^S_j)\tilde v_{j,l},
\end{multline}
where we conveniently chose to write Duhamel on $S$ for the first term
and Duhamel on $\Omega$ for the second one, which allows to commute
the flow under the time integral. Denote by $v_{j,l,m}$ the first term in the second line of
\eqref{formevjld} by $v_{j,l,f}$ the second one and $v_{j,l,s}$ the
last one. We deal with them
separately. To estimate the $L^5_x L^2_t $ norm of the $v_{j,l,f}$ we
notice that its support is far from the boundary: as such, estimates
on the $L^5_x L^2_t $ norm will follow from Section
\ref{sectfarfrombound}. Indeed, we get
\begin{equation}\label{est1moinschi}
\|(1-\chi_{1})\tilde{\Delta}_{j}e^{i(t-s)\Delta_{D}}V_{j,l}\|_{L^{5}_x L^{2}_t}\lesssim\|\tilde{\Delta}_{j}V_{j,l}\|_{\dot{H}^{-1/10}(\Omega)}\simeq 2^{-\frac{j}{10}}\|\tilde{\Delta}_{j}V_{j,l}\|_{L^{2}(\Omega)}.
\end{equation}
We then apply the Minkowski inequality to deduce
\begin{multline}
\|(1-\chi_{1})\int_{h(l-1/2)\pi}^{t}
\tilde{\Delta}_{j}e^{i(t-s)\Delta_{D}}V_{j,l}(s,x)ds\|_{L^{5}_{x}L^{2}_{t}}\\
\leq
 2^{-j/2}(\int_{I_{j,l}}\|(1-\chi_{1})\tilde{\Delta}_{j}e^{i(t-s)\Delta_{D}}V_{j,l}(s,.)\|^{2}_{L^{5}(\Omega)L^{2}(I_{j,l})}ds)^{1/2},
\end{multline}
where we denoted $I_{j,l}=[2^{-j}(l-1/2)\pi,2^{-j}(l+1)\pi]$ and we used the Cauchy-Schwartz inequality.
Using \eqref{est1moinschi} we finally get
\begin{equation}\label{estvjlfnorm}
\|v_{j,l,f}\|_{L^{5}(\Omega)L^{2}(I_{j,l})}\leq 2^{-j(1/2+1/10)}\|\tilde{\Delta}_{j}V_{j,l}\|_{L^{2}(I_{j,l})L^{2}(\Omega)}.
\end{equation}

To estimate the $L^5_x L^2_t $ norm of the main contribution $v_{j,l,m}$ we need the following:
\begin{prop}\label{propssinai}
Let $j\geq 0$, $I_{j}=(-\pi 2^{-j},\pi 2^{-j})$, $\tilde{\chi}\in
C^{\infty}_{0}(\mathbb{R}^{3})$ be supported near $\partial\Omega$ and
$V_{0}\in L^{2}(\Omega)$. Then there exists $C>0$ independent of $j$
such that for the solution
$e^{it\Delta_{S}}\tilde{\Delta}^S_j  \tilde{\chi}V_{0}$ of
the linear Schr\"{o}dinger equation on $S$ with initial data
$\tilde{\Delta}^S_j \tilde{\chi} V_{0}$ we have
\begin{equation}\label{solu}
\|e^{it\Delta_{S}}\tilde{\Delta}^S_j  \tilde{\chi}V_{0}\|_{L^{5}(S)L^{2}_t (I_{j})}\leq C2^{-\frac{j}{10}}\|\tilde{\Delta}^S_j \tilde{\chi}V_{0}\|_{L^{2}(S)}.
\end{equation}
\end{prop}
We postpone the proof of Proposition \ref{propssinai} to Subsection \ref{proofP31}.

Using the fact that $v_{j,l}$ is supported in time in $I_{j,l}=[2^{-j}(l-1/2)\pi,2^{-j}(l+1)\pi]$, the Minkowski inequality, Proposition
\ref{propssinai} with $\tilde{\chi}=1$ on the support of $\chi$ and
with $V_{0}=V_{j,l}$, and since
$\tilde{\chi}_{1}v_{j,l,m}=v_{j,l,m}$ for any $\tilde{\chi}_{1}\in C^{\infty}(\mathbb{R}^{3})$ with $\tilde{\chi}_{1}=1$ on the support of $\chi_{1}$, we obtain
\begin{align}\label{estvjlm}
\|v_{j,l,m}\|_{L^5(\Omega) L^2(I_{j,l}) }  =& \|\tilde{\chi}_{1}v_{j,l,m}\|_{L^5(\Omega)
  L^{2}(I_{j,l})}= \|v_{j,l,m}\|_{L^{5}(S)L^{2}(I_{j,l})} \nonumber\\
  \leq &
 \int_{2^{-j}(l-1)\pi}^{2^{-j}(l+1)\pi}\|e^{i(t-s)\Delta_{S}}\tilde{\Delta}^S_j
 V_{j,l}(s,.)\|_{L^5(S) L^{2}(I_{j,l})}ds\nonumber\\
\leq &
 2^{-\frac{j}{10}}\int_{I_{j,l}}\|\tilde{\Delta}^S_j
 V_{j,l}(s)\|_{L^{2}(S)}ds\nonumber\\
  \leq & 2^{-\frac{j}{10}}\int_{I_{j,l}}\|
\tilde{\chi}V_{j,l}(s)\|_{L^{2}(S)}ds\nonumber\\
  \leq  &2^{-\frac{j}{10}}\int_{I_{j,l}}\|\tilde\chi
  V_{j,l}(s)\|_{L^{2}(\Omega)}ds
\end{align} 
where we used again $V_{j,l}=\tilde{\chi}V_{j,l}$ to switch $S$ and
$\Omega$ and continuity of $\Delta_j^S$ on $L^2(S)$. 
Using the Cauchy-Schwartz inequality in \eqref{estvjlm} yields
\begin{equation}
  \label{estvhl}
\|v_{j,l,m}\|_{L^5(\Omega)L^2(I_{j,l}) } \lesssim
2^{-j(1/2+1/10)}\|V_{j,l}\|_{L^{2}(I_{j,l})L^{2}(\Omega)}
\end{equation}
We deal with the right handside in \eqref{estvhl}. Using the explicit expression of $V_{j,l}$ given in \eqref{grandvhl},
\begin{multline}\label{estvjlmfar}
\|V_{j,l}(s)\|_{L^{2}(I_{j,l})L^{2}(\Omega)}\lesssim (\|\varphi_{l}(2^{j}t)[\Delta_{D},\chi]\Delta_{j}u_L\|_{L^{2}(I_{j,l})L^{2}(\Omega)}\\
{}+2^j\|\varphi'_{l}(2^{j}t)\chi\Delta_{j}u_L\|_{L^{2}(I_{j,l})L^{2}(\Omega)}).
\end{multline}
As $[\Delta_D,\chi]$ is bounded from $H^1_0$ to $L^2$, we get
\begin{equation}\label{estvjlmclose}
\|{\tilde \Delta_j}V_{j,l} \|_{L^{2}(I_{j,l})L^{2}(\Omega)} \lesssim 
\|\chi_1 \Delta_j
u_L\|_{L^2(I_{j,l})H^1_0(\Omega)} +2^{j}\|\chi \Delta_j
u_L\|_{L^{2}(I_{j,l})L^{2}(\Omega)}
\end{equation}
Let us recall the following local smoothing result on a non trapping domain:
\begin{lemma}\label{lemnic}(Burq, G\'erard, Tzvetkov \cite[Prop.2.7]{bgt03})
Assume that $\Omega=\mathbb{R}^{3}\setminus\Theta$, where
$\Theta\neq\emptyset$ is a non-trapping obstacle. Then, for every
$\tilde{\chi}\in C^{\infty}_{0}(\mathbb{R}^{3})$, and $\sigma\in
[-1/2,1]$, 
\begin{equation}\label{smoothi}
\|\tilde{\chi} \Delta_j u_L\|_{L^{2}(\R,\dot H^{\sigma+1/2}(\Omega))}\leq C\|\Delta_j u_{0}\|_{H^\sigma(\Omega)},
\end{equation}
where, as usual, $u_L(t,x)=e^{-it\Delta_{D}}u_{0}(x)$.
\end{lemma}
We now turn to the difference term  $v_{j,l,s}$ and prove a smoothing
lemma.

\begin{lemma}\label{lemschwartzrest} Let $\chi_{1}\in C^{\infty}_{0}(\mathbb{R}^{n})$ be equal to $1$ on a fixed neighborhood of
 the support of $\tilde{\chi}$. Then we have for all $N\in\mathbb{N}$,
\begin{equation}\label{estvhleeees}
\|v_{j,l,s}\|_{L^{5}(\Omega)L^2(I_{j,l})}\leq C_{N}2^{-Nj}\|V_{j,l}(x,s)\|_{L^{2}(I_{j,l} ,L^2(\Omega))}.
\end{equation}
\end{lemma}
In order to prove the lemma, one would like to rewrite
$\tilde{\Delta}_j=\tilde{\psi}(2^{-2j} \Delta_D)$ as a solution of the wave equation,
using $h=2^{-j}$ as a time. Then the finite speed of propagation would let
us switch $\Delta_D$ and $\Delta_S$. However the inverse Fourier
transform (in $|\xi|$) of $\Psi(|\xi|)=\tilde{\psi}(|\xi|^2)$ is only Schwartz class, rather
than compactly supported. The tails will eventually account for the
right handside of \eqref{estvhleeees}. We now turn to the details: let
$\varphi_0,\varphi(y)$ be even, compactly supported ($\varphi(y)$ away from zero) and
such that 
$$
\varphi_0(y)+\sum_{k\geq 1} \varphi(2^{-k}y)=1.
$$
We decompose $\hat \Psi(y)$ using this resolution of the identity, and
set with obvious notations
$$
\Psi(|\xi|)=\sum_{k\in \N} \phi_k(|\xi|),
$$
where the $\phi_k$ have good bounds, say $\hat \phi_0\in L^\infty$ and
for $k\geq 1$
\begin{equation}
  \label{eq:decrap}
  \forall N\in \N,\,\,\, \|\hat\phi_k\|_\infty =\|\hat\Psi(y) \varphi(2^{-k}y)\|_\infty\leq C_N 2^{-kN}.
\end{equation}
At fixed $k$, we write (abusing notation and letting $\Delta$ be
either $\Delta_D$ or $\Delta_S$)
$$
\phi_k(h \sqrt{-\Delta})\tilde \chi \tilde{v}_{j,l}=\frac 1 {2\pi} \int
e^{i y h \sqrt{-\Delta}} \tilde \chi(x) \tilde{v}_{j,l}(x) \hat \phi_k(y)\,dy.
$$
Notice that $\phi_k(y)$ is compactly supported, in fact its support is
roughly $|y|\in [2^{k-1},2^{k+1}]$. As such the $y$ integral is a time
average of half-wave operators, which have finite speed of
propagation. Therefore if the \og time\fg{} $|yh|\leq 1$, we can add
another cut-off function $\chi_1$ which is equal to one on the domain of
dependency of $\tilde \chi$ on this time scale, and such that $\chi_1$
is indifferently defined on $S$ or $\Omega$: namely, for $k\lesssim j$,
\begin{eqnarray}
  \label{eq:segald}
\phi_k(h \sqrt{-\Delta_S})\tilde \chi \tilde{v}_{j,l} & =
&\chi_1(x)\phi_k(h \sqrt{-\Delta_S})\tilde \chi \tilde{v}_{j,l}\nonumber\\
 & = & \chi_1(x) \frac 1 {2\pi} \int
e^{i y h \sqrt{-\Delta}} \tilde \chi(x) \tilde{v}_{j,l}(x) \hat
\phi_k(y)\,dy\nonumber,\\
\phi_k(2^{-j} \sqrt{-\Delta_S})\tilde \chi \tilde{v}_{j,l} & = & \chi_1(x)\phi_k(2^{-j} \sqrt{-\Delta_D})\tilde \chi \tilde{v}_{j,l}.
\end{eqnarray}
From this identity, we obtain
\begin{equation}
  \label{eq:diffsd}
  v_{j,l,s}=\chi_1(x) \sum_{j\lesssim k}
(\phi_k(2^{-j}\sqrt{-\Delta_D})-\phi_k(2^{-j}\sqrt{-\Delta_S})) \tilde \chi(x)\tilde{v}_{j,l}.
\end{equation}
At this point the difference in \eqref{eq:diffsd} is irrelevant and we
estimate both terms using Sobolev embedding and energy
estimates. Abusing notations, with $\Delta\in \{\Delta_D,\Delta_S\}$, we have
\begin{align*}
\|  \chi_1 \phi_k(2^{-j}\sqrt{-\Delta})\tilde \chi
\tilde{v}_{j,l}\|_{L^5(\Omega)L^2_t (I_{j,l})} \leq & \|  \chi_1 \phi_k(2^{-j}\sqrt{-\Delta})\tilde \chi
\tilde{v}_{j,l}\|_{L^2_t (I_{j,l})L^5(\Omega)} \\
 \leq & 2^{-\frac j 2} \|  \chi_1 \phi_k(2^{-j}\sqrt{-\Delta})\tilde \chi
\tilde{v}_{j,l}\|_{L^\infty_t (I_{j,l})L^5(\Omega)} \\
\lesssim  & 2^{-\frac j 2} \|   \phi_k(2^{-j}\sqrt{-\Delta})\tilde \chi
\tilde{v}_{j,l}\|_{L^\infty_t (I_{j,l})H^{\frac 1 2 }(\Omega)} \\
\lesssim  & C_N 2^{-\frac j 2-kN} \|  \tilde \chi
\tilde{v}_{j,l}\|_{L^\infty_t (I_{j,l})H^{\frac 1 2}(\Omega)} 
\end{align*}
where we used Minkowski, H\"older, (non sharp !) Sobolev and
\eqref{eq:decrap}. Finally, by the dual estimate of \eqref{smoothi},
$$
\|  \tilde{v}_{j,l}\|_{L^\infty_t (I_{j,l})H^{\frac 1 {2}}(\Omega)} \lesssim
\|V_{j,l}\|_{L^2_t (I_{j,l},L^2(\Omega))}.
$$
Summing in $k$ and relabeling $N$, we have
\begin{equation}
  \label{eq:ouf}
  \|  v_{j,l,s}\|_{L^5(\Omega)L^2_t (I_{j,l})} \leq C_N 2^{-j N}\|V_{j,l}\|_{L^2_t (I_{j,l},L^{2}(\Omega))},
\end{equation}
which concludes the proof of the lemma.

Using this lemma and \eqref{estvjlmclose}, we get for $v_{j,l,s}$ an
estimate which matches \eqref{estvhl}: picking $N=1$ is enough. From there,
 using \eqref{estvh}, \eqref{estvjlfnorm}, \eqref{estvhl}, we write
\begin{align*}
\|{\tilde \Delta_j}\chi \Delta_j  u_L\|^{2}_{L^5(\Omega) L^2_{t} }\lesssim &
2^{-2j(\frac 1 2 +\frac 1 {10})}\sum_{l\in\mathbb{Z}} \|{\tilde
  \Delta_j}V_{j,l}(s)\|^{2}_{L^{2}(I_{j,l})L^{2}(\Omega)}\\
\lesssim & 
2^{-2j(\frac{1} 2+\frac 1{10})}\sum_{l\in\mathbb{Z}}(\|\tilde{\chi}\Delta_j
u_L\|^{2}_{L^{2}(I_{j,l})H^{1}_0(\Omega)}+2^{2j}\|\tilde{\chi}\Delta_j
u_L\|^{2}_{L^{2}(I_{j,l})L^2(\Omega)})\\
 \lesssim &
2^{-\frac{2 j}{10}}( 2^{-j}\|{\tilde
  \Delta_j}u_{0}\|^{2}_{\dot H^\frac 1 2(\Omega)}+2^{j} \|{\tilde
  \Delta_j}u_{0}\|^{2}_{\dot H^{-\frac 1 2}(\Omega)})\\
 \lesssim &
2^{-\frac{2 j}{10}}( \|{\tilde
  \Delta_j}u_{0}\|^{2}_{L^2(\Omega)},
\end{align*}
which is the desired result.
\subsubsection{End of the proof of Theorem \ref{thmestim}}
Until now we have prove Theorem \ref{thmestim} only for $q=2$. We shall use the Gagliardo-Nirenberg inequality in order to deduce \eqref{l5l2} for every $q\geq 2$. We have
\[
\|\Delta_j u_{L}\|_{L^{\infty}_{t}}\lesssim\|\Delta_j u_{L}\|_{L^{2}_{t}}^{1/2}\|\Delta_j \partial_{t}u_{L}\|^{1/2}_{L^{2}_{t}}.
\]
which gives, taking the $L^5_x $ norms and using the Cauchy-Schwartz inequality
\begin{equation}\label{estpart1}
\|\Delta_j u_{L}\|^{5}_{L^5_x L^{\infty}_{t}}\lesssim\|\Delta_j u_{L}\|^{5/2}_{L^{5}L^{2}_{t}}\|\Delta_j \partial_{t}u_{L}\|^{5/2}_{L^5_x L^{2}_{t}}.
\end{equation}
It remains to estimate $\|\Delta_j \partial_{t}u_{L}\|_{L^5_x L^{2}_{t}}$: notice that since $u_{L}=e^{-it\Delta_{D}}u_{0}$
\[
\Delta_j \partial_{t}u_{L}=-i\Delta_D \Delta_{j}u_{L}=i 2^{2j}\tilde
\Delta_j u_{L},
\]
where $\tilde\Delta_j$ is defined with $\psi_{1}(x)=x\psi(x)\in
C^{\infty}_{0}(\mathbb{R}\setminus\{0\})$. Therefore
\begin{equation}\label{estpart2}
 \|\Delta_j \partial_{t}u_{L}\|_{L^5_x L^{\infty}_{t}}\leq
 C2^{j(2-1/10)}\|\tilde \Delta_j u_{0}\|_{L^{2}(\Omega)},
\end{equation}
consequently 
\[
\|\Delta_j \partial_{t}u_{L}\|_{L^5_x L^{q}_{t}}\leq C2^{-j(2/q-9/10)}\|\Delta_j u_{0}\|_{L^{2}(\Omega)}
\]
and Theorem \ref{thmestim} is proved.

\subsection{Proof of Theorems \ref{thmestimin} and \ref{thmestiminbis}}
 We recall a lemma due to Christ and
Kiselev \cite{chki01}. We state the corollary we will use, with only
the time variable: we refer to \cite{buplck} for a simple direct proof
of all the different cases we use, with Banach-valued $L^p_t(B)$
spaces or $B(L^p_t)$. Its use in the context of reversed norms
$L^q_x(L^p_t)$ goes back to \cite{pl02} and it greatly simplifies
obtaining inhomogeneous estimates from homogeneous ones.
\begin{lemma}(Christ and Kiselev \cite{chki01})
Consider a bounded operator 
\[
T:L^{r}(\R)\rightarrow L^{q}(\R)
\]
given by a locally integrable kernel $K(t,s)$. Suppose that
$r<q$. Then the restricted operator

\[
{T}_Rf(t)=\int_{s<t} K(t,s)f(s)ds
\]
is bounded from $L^{r}(\R)$ to $L^{q}(\R)$ and
\[
\|{T}_R\|_{L^{r}(\mathbb{R})\rightarrow L^{q}(\R)}\leq C (1-2^{-(1/q-1/r)})^{-1}\|T\|_{L^{r}(\R)\rightarrow L^{q}(\R)}.
\]
\end{lemma} 
From the lemma, the proof of the inhomogeneous set of estimates in
Theorem \ref{thmestimin} is routine from the homogeneous estimates in
Theorem \ref{thmestim} and the Duhamel formula. Combining both
homogeneous and inhomogeneous estimates yields Theorem \ref{thmestiminbis}.

\subsection{Proof of Proposition \ref{propssinai}}\label{proofP31}
Let $S$ denote the compact domain defined above.
Recall $(e_{n})_{n}$ is the eigenbasis of $L^{2}(S)$ consisting of
eigenfunctions of $-\Delta_{S}$ associated to the eigenvalues
$\lambda^{2}_{n}$. 
Following \cite{bulepl07}, we define an abstract self adjoint operator on $L^{2}(S)$ as follows
\[
A_{h} (e_{n}):=-[h\lambda^{2}_{n}]e_{n},
\]
where $[\lambda]$ is the integer part of $\lambda$. Notice that in some sense $A_{h}="[h\Delta_{S}]"$. We first need to establish estimates for the linear Schr\"odinger equation on the compact domain $S$ with spectrally localized initial data. 

We now set $h=2^{-j}$ and state estimates on the evolution equation
where $h\Delta_S$ is replaced by $A_h$.
\begin{lemma}\label{lemutil}
Let $0<h\leq 1$, $q\geq 2$, $I_{h}=(-\pi h,\pi h)$,
$\tilde{\chi}\in C^{\infty}_{0}(\mathbb{R}^{3})$ be supported near
$\partial\Omega$ and $V_{0}\in L^{2}(\Omega)$. There exists $C>0$
independent of $h$ such that 
\begin{equation}\label{solu1}
\|e^{i \frac{t}{h}A_{h}}\tilde{\Delta}^S_j \tilde{\chi} V_{0}\|_{L^{5}(S)L^{q}(I_{h})}\leq Ch^{2/q-9/10}\|\tilde{\Delta}^S_j \tilde{\chi}V_{0}\|_{L^{2}(S)}.
\end{equation}
\end{lemma}
We postpone the proof of Lemma \ref{lemutil} and proceed with the
proof of Proposition \ref{propssinai}. Denote by $V_{h}(t,x):=e^{it\Delta_{S}}\tilde{\Delta}^S_j \tilde{\chi}V_{0}(x)$, then
\begin{equation*}
  (ih\partial_{t}+A_{h})V_{h}=(ih\partial_{t}+h\Delta_{S})V_{h}+(A_{h}-h\Delta_{S})V_{h}
=(A_{h}-h\Delta_{S})e^{it\Delta_{S}}\tilde{\Delta}^S_j \tilde{\chi} V_{0}.
\end{equation*}
Writing Duhamel formula for $V_{h}$ yields
\begin{equation}
  \label{duh}
V_{h}(t,x)=e^{i\frac{t}{h}A_{h}}\tilde{\Delta}^S_j \tilde{\chi} V_{0}(x)-\frac{i}{h}\int_{0}^{t}e^{i\frac{(t-s)}{h}A_{h}}(A_{h}-h\Delta_{S})e^{is\Delta_{S}}\tilde{\Delta}^S_j \tilde{\chi} V_{0}(x)ds.
\end{equation}
Using \eqref{solu1} with $q=2$, \eqref{duh}, the Minkowski inequality
and boundedness of the operator 
$$
\|e^{i\frac{t}{h}A_{h}}\tilde{\Delta}^S_j\|_{L^{2}(S)\rightarrow
  L^{5}(S)L^{2}(I_{h})}\lesssim 2^{-\frac{j}{10}} \sim h^{1/10}
$$
(which follows from the proof of Lemma
\ref{lemutil}), we obtain
\begin{multline}
  \label{esttcvc}
\|e^{it\Delta_{S}}\tilde{\Delta}^S_j
\tilde{\chi}V_{0}\|_{L^{5}(S)L^{2}(I_{h})}\lesssim h^{\frac 1{10}}\Big(\|\tilde{\Delta}^S_j \tilde{\chi} V_{0}\|_{L^{2}(S)}\\
{}+\frac{1}{h}\|(A_{h}-
h\Delta_{S})e^{is\Delta_{S}}\tilde{\Delta}^S_j \tilde{\chi} V_{0}\|_{L^{1}(-h\pi,h\pi)L^{2}(S)}\Big),
\end{multline}
where to estimate the second term in the right hand side of
\eqref{duh} we used the fact that $A_{h}$ commutes with the spectral
localization $\tilde{\Delta}^S_j $.
Changing variables $s=h\tau$ in the second term in the right hand side
of \eqref{esttcvc} yields
\begin{multline}
\label{esahdh}
\frac{1}{h}\|(A_{h}-
h\Delta_{S})e^{is\Delta_{S}}\tilde{\Delta}^S_j \tilde{\chi} V_{0}\|_{L^{1}(-h\pi,h\pi)L^{2}(S)}=
\int_{-\pi}^{\pi}\|(A_{h}-
h\Delta_{S})e^{i\tau h\Delta_{S}}\tilde{\Delta}^S_j \tilde{\chi}
V_{0}\|_{L^{2}(S)}d\tau \\
\lesssim 2\pi \|\tilde{\Delta}^S_j \tilde{\chi} V_{0}\|_{L^{2}(S)},
\end{multline}
where we used the fact that the operator $(A_{h}-h\Delta_{S})$ is
bounded on $L^{2}(S)$ and the mass conservation of the linear
Schr\"odinger flow. If follows from \eqref{esttcvc} and \eqref{esahdh}
that
\[
\|e^{it\Delta_{S}}\tilde{\Delta}^S_j
\tilde{\chi}V_{0}\|_{L^{5}(S)L^{2}(I_{h})}\lesssim h^{1/10}\|\tilde{\Delta}^S_j \tilde{\chi}V_{0}\|_{L^{2}(S)},
\]
which ends the proof of Proposition \ref{propssinai}.

We now return to Lemma \ref{lemutil} for the rest of this section. Writing $\tilde{\Delta}^S_j V_{0}=\sum_{n}\tilde{\psi}(h^{2}\lambda^{2}_{n})V_{\lambda_{n}}e_{n}$, we decompose (for $0<h\leq 1/4$)
\[
e^{i\frac{t}{h}A_{h}}\tilde{\Delta}^S_j
V_{0}(t,x)=\sum_{k\in\mathbb{N}}e^{i\frac{t}{h}k} v_k(x)
\]
with
\[
v_k(x)=\sum_{\lambda=(k2^j)^{1/2}}^{((k+1)2^j)^{1/2}-1}\sum_{\lambda_{n}\in [\lambda,\lambda+1)}\tilde{\Psi}(h^{2}\lambda^{2}_{n})V_{\lambda_{n}}e_{n}=\sum_{\lambda=(k2^j)^{1/2}}^{((k+1)2^j)^{1/2}-1}\Pi_{\lambda}(\tilde{\Delta}^S_j
V_{0}),
\]
 where $\Pi_{\lambda}$
denotes the
spectral projector
$\Pi_{\lambda}=1_{\sqrt{-\Delta_{S}}\in[\lambda,\lambda+1)}$. Let us estimate the $L^{5}(S)L^{q}(I_{h})$ norm of $e^{i\frac{t}{h}A_{h}}\tilde{\Delta}^S_j  V_{0}$: 
\begin{align*}
  \|e^{i\frac{t}{h}A_{h}}\tilde{\Delta}^S_j
  V_{0}\|^{2}_{L^{5}(S)L^{q}(I_{h})} & \lesssim h^{2/q}\|
  \|e^{isA_{h}}\tilde{\Delta}^S_j
  V_{0}\|^{2}_{L^{q}_s(-\pi,\pi)}\|_{L^{5/2}(S)}\\ & \lesssim
h^{2/q}\| \|e^{isA_{h}}\tilde{\Delta}^S_j
V_{0}\|^{2}_{H^{1/2-1/q}(s\in(-\pi,\pi))}\|_{L^{5/2}(S)}\\
 & \lesssim
h^{2/q}\|\sum_{k\in\mathbb{N}}(1+k)^{2(\frac 1 2-\frac 1 q)}\|
e^{isk}v_k(x)\|^{2}_{L^{2}_s(-\pi,\pi)}\|_{L^{5/2}(S)}\\
 & \lesssim 
h^{2/q}\sum_{k\in\mathbb{N}}(1+k)^{1-2/q}\|e^{isk}v_k(x)\|^{2}_{L^{5}(S)L^{2}(-\pi,\pi)}\\
 & \lesssim
h^{2/q}\sum_{k\in\mathbb{N}}(1+k)^{1-2/q}\|e^{isk}v_k(x)\|^{2}_{L^{2}(-\pi,\pi)L^{5}(S)},
\end{align*}
where we used Sobolev injection in the time variable
$H^{1/2-1/q}\subset L^{q}$ and Plancherel in time. We
recall a result of \cite{smso06} of Smith and Sogge on the
spectral projector $\Pi_{\lambda}$:
\begin{thm}\label{thmprojscl}(Smith and Sogge \cite{smso06}) Let $S$ be a compact manifold of dimension $3$, then
\[
\|\Pi_{\lambda}\|_{L^{2}(S)\rightarrow L^{5}(S)}\leq \lambda^{2/5}.
\]
\end{thm}
Using Theorem \ref{thmprojscl} we have
\begin{align*}
  \|e^{i\frac{t}{h}A_{h}}\tilde{\Delta}^S_j
  V_{0}\|^{2}_{L^{5}(S)L^{q}(I_{h})} & \lesssim
h^{2/q}\sum_{1/4h-1\leq k\leq
  4/h}(1+k)^{1-2/q+4/5}\|\tilde{\Delta}^S_j V_{0}\|^{2}_{L^{2}(S)}\\
 & \lesssim
\sum_{hk\in [1/4,4]}k^{1-4/q+4/5}\|\tilde{\Delta}^S_j
V_{0}\|^{2}_{L^{2}(S)}\\
 & \lesssim
\|\tilde{\Delta}^S_j V_{0}\|^{2}_{\dot{H}^{2/q-9/10}(S)},
\end{align*}
since for $hk>4$ or $h(k+1)<1/4$ and $\lambda_{n}\in [(k2^j)^{1/2},((k+1)2^j)^{1/2})$ we have $\tilde{\Psi}(h^{2}\lambda^{2}_{n})=0$ and on the other hand for these values of $k$ we have
\[
k/\sqrt{2}\leq(k2^j)^{1/2}\leq\lambda_{n}\leq ((k+1)2^j)^{1/2}\leq\sqrt{2}(k+1),\quad h\leq 5(k+1)^{-1}.
\]
This completes the proof of Lemma \ref{lemutil}.

\section{Local existence}
In this section we prove Theorem \ref{thmleq}. 
\begin{dfn}
Let $u\in \mathcal{S}'(\mathbb{R}\times \Omega)$ and let
$\Delta_j=\psi(-2^{-2j}\Delta_{D})$ be a spectral localization with respect to
the Dirichlet Laplacian $\Delta_D$ in the $x$ variable, such that
$\sum_j \Delta_j =Id$ and let $S_j=\sum_{k<j} \Delta_j$. 
We introduce the "Banach valued" Besov space  $
\BL s p q r$ as follows: we say that $u\in
\BL s p q r$ if
\[
\Big(2^{js}\|\Delta_j u\|_{L^{p}_{x}L^{r}_{t}}\Big)\in l^{q},
\]
and $\sum_j \Delta_j f$ converges to $f$ in $\mathcal{S}'$. If
${L}^r_t$ is replaced by ${L}^r_T$, the time
integration is meant to be over $(-T,T)$. Moreover, when $s<0$,
$\Delta_j$ may be replaced by $S_j$ in the norm and both norms are equivalent.
\end{dfn}
Consider $u_0\in \dot H^1_0$ and $u_L$ the solution to the linear
equation \eqref{LS}. Applying Theorem \ref{thmestim} with $q=2,5$ and taking $s=1$ in the definition above we obtain 
\[
u_{L}\in \BL {1+\frac 1 {10}} 5 2 {2}\cap \BL {\frac 1 2} 5 2 5 \,\,\text{
  and 
  } \,\, \partial_t u_L \in \BL {-\frac 3 2} 5 2 5.
\]
From this, by Gagliardo-Nirenberg in the time variable, one should have
\[
u_L\in\BL 1 5 2 {\frac{20} 9}\cap \BL {3/20} 5 2 {40}\subset L^{20/3}_xL^{40}_t,
\]
and consequently 
\[
u^{4}_L\in L^{5/3}_x L^{10}_{t} \text{ as well as  } |u_L|^4 u_L \in 
\BL 1 {\frac 5 4} 2 {\frac{20}{11}}
\]
which should be enough to iterate. However, our spaces are Banach
valued Besov spaces (if one sees time as a parametrer) and justifying
Berstein-like inequalities and Sobolev embedding is not entirely
trivial (but doable, using the estimates from \cite{ip08}). We choose an apparently complicated space in order
to set up the fixed point, but the little gain in regularity from the
smoothing estimate will turn out to be crucial for subcritical
scattering.
\begin{rmq}
  By this choice, we only restrict the uniqueness class. It is likely
  that one may prove a better result, but there is no immediate
  benefit in the present setting, except proving additional
  estimates. We retained, however, the uniqueness class that would
  be provided by the argument above in the Theorems'statements. Another remark is that one may
  dispense with the use of Lemma \ref{ST}, miss the endpoint $q=2$ and
  still get the exact same nonlinear results, as there is room (due to
  the use of Sobolev embedding) in all mapping estimates. Moreover, as
  soon as we use an estimate with a (however small) gain in
  regularity, we do not need Lemma \ref{Sob}, as we could use a
  simpler embedding in a Besov space of negative regularity and play
  regularities against each other. In fact, in the same spirit as
  \cite{pl02} one could replace the critical Sobolev norm by a Besov
  norm $\dot B^{s_p,\infty}_2$.
\end{rmq}
For $T>0$ let
\begin{equation}
  \label{eq:xt}
X_{T}:= \{u \,|\, u\in \BLT {1+\frac 1 {10}} 5 2 {2}\cap \BLT {\frac 1 2} 5 2 5 \,\,\text{
  and 
  } \,\, \partial_t u \in \BLT {-\frac 3 2} 5 2 5\}.
\end{equation}
 and for $u\in X_{T}$ set $F(u):=|u|^{4}u$.   
\begin{prop}
Define a nonlinear map $\phi$ as follows,
\[
\phi(u)(t):=\int_{s<t} e^{i(t-s)\Delta_{D}}F(u(s))ds.
\]
Then
\begin{equation}
  \label{eq:map5}
\|\phi(u)\|_{C_T(\dot H^1_0)}+\|\phi(u)\|_{X_{T}}\lesssim\|F(u)\|_{\BLT
  1 {5/4} 2 {20/11}}\lesssim \|u\|^5_{X_{T}},
\end{equation}
and
\begin{equation}
    \label{eq:map5d}
\|\phi(u)-\phi(v)\|_{X_{T}}\lesssim\|F(u)-F(v)\|_{\BLT 1 {5/4} 2
  {20/11}}\lesssim\|u-v\|_{X_T}(\|u\|_{X_T}+\|v\|_{X_T})^{4}.
\end{equation}
\end{prop}
The estimate for the inhomogeneous problem writes
\[
\|\int e^{-is\Delta_{D}}F\|_{L^{2}_x}\leq C\|F\|_{\BL 0{5/4} 2 {20/11}},
\]
Shifting the regularity to $s=1$ and using the Christ-Kiselev lemma
provides the first step of both estimates \ref{eq:map5} and
\ref{eq:map5d}. Now, Lemma \ref{diffpara} in the Appendix provides the
nonlinear part of both estimates (note however that, as $p=5$ is an
integer, one could prove directly the nonlinear mappings by product
rules).

One may now set up the usual fixed point argument in $X_{T}$ if $T$ is
sufficiently small of if the data is small. This concludes the proof
of Theorem \ref{thmleq} (scattering for small data follows the usual
way from the global in time space-time estimates).

We now consider local wellposedness for $p<5$, e.g. Theorem \ref{thmleqsub}. The critical
Sobolev exponent w.r.t. scaling is $s_p=3/2-2/(p-1)$. We aim at
setting up a contraction argument in a small ball of
\begin{equation}
  \label{eq:xtc}
X_{T}:= \{u \,|\, u\in \BLT {s_p+\frac 1 {10}} 5 2 {2}\cap \BLT {s_p-\frac 1 4} 4 2 4 \,\,\text{
  and 
  } \,\, \partial_t u \in \BLT {s_p-\frac 1 4 -2} 4 2 4\}.
\end{equation}
The important fact (if we were to ignore issues with Banach valued
Besov spaces) would be that $X_{T}\subset \BLT{s_p} 5 2{20/9}\cap L^{5(p-1)/3}_x L^{10(p-1)}_{T}$.  
 \begin{rmq}
   Some numerology is in order: if one were only to have the
   $L^5_x L^2_t$ smoothing estimate and use Sobolev (in time and
   in space), it would require $5(p-1)/3\geq 5$, namely $p\geq
   4$. However, we have the Strichartz estimate from \cite{plve08},
   which allows $5(p-1)/3\geq 4$, or $p\geq 3+2/5$.
 \end{rmq}
Again from the Appendix, the nonlinear mapping verifies
\[
\| F(u)-F(v)\|_{\BLT {s_p} {5/4} 2 {20/11}}\lesssim
\|u-v\|_{X_T}(\|u\|^{p-1}_{X_T}+\|v\|^{p-1}_{X_T})
\]
and existence and uniqueness follow by fixed point again.

\subsection{Scattering for $3+2/5< p<5$ }
We now deal with scattering in the same range of $p\in (3+2/5,5)$: from
\cite{plve08}, we have an a priori bound
$$
\|S_j u\|^4_{L^4_tL^4_x}\lesssim\| u\|^4_{L^4_tL^4_x}\lesssim \|u_0\|^3_{L^2_x}
  \sup_t \|u\|_{H^1_0}\leq M^\frac 3 2 E^\frac 1 2,
$$
where $M$ and $E$ are the conserved charge and hamiltonian,
\begin{equation}\label{massenerg}
M=\int_\Omega |u|^2 \,dx \text{ and }  E=\int_\Omega |\nabla u|^2
+\frac{2}{p+1} |u|^{p+1}\,dx.
\end{equation}
Notice how this estimate is below the critical scaling $s_p$,
as the RHS regularity is $s=1/4$. From the energy a priori bound and
Sobolev embedding, one has on the other hand
$$
\|S_j u\|_{L^\infty_{t,x}} \lesssim   2^{\frac j 2}\sup_t
\|u\|_{H^1_0}\lesssim 2^{\frac j 2} E^\frac 1 2.
$$
Interpolating between the two bounds to get the right scaling yields, 
\begin{equation}
  \label{eq:scat}
  \|S_j u\|_{L^q_{t,x}} \lesssim   C(M,E) 2^{j(\frac 1 2-\frac{5-p}{3(p-1)})},
\end{equation}
where $1/q=(5-p)/6(p-1)$. In order to proceed with the usual scattering argument, we need to
revisit the fixed point, or more precisely the nonlinear estimate on
$F(u)$: indeed, if we wish to use \eqref{eq:scat}, even at a power
$\varepsilon$, we cannot afford to use the same regularity on both
sides of the Duhamel formula. Fortunately, we have off diagonal
inhomogeneous estimates, e.g.
\[
\|\int e^{i(t-s)\Delta_{D}}F\|_{\BL{s_p} 5 2{20/9}\cap \BL{s_p-3/4} 4
  2 4}
\leq C\|F(u)\|_{\BL{s_p- \frac 1 {10}}{5/4} 2 {2}}.
\]
In order to evaluate $F(u)$, one needs to place the $S_j u$ factors in
such a way that
$$
\|(S_j u)^{p-1}\|_{L^{ 5/3 }_x L^{20}_t}\lesssim 2^{\frac j
  {10}}.
$$
However, we have from \eqref{eq:scat}
\begin{equation}
  \|(\Delta_j u)^{p-1}\|_{L^{\frac 6 {5-p}}_{t,x}} \lesssim   C(M,E) 2^{j(\frac {5p-13}{6})},
\end{equation}
and $6/(5-p)> 5/3$. As such, one may interpolate with
$$ 
\|\Delta_j u\|_{L^4_x L^4_t}\lesssim 2^{-j(s_p-\frac 1 4)},
$$
to get (after Sobolev embedding)
$$
\|(\Delta_j u)^{p-1}\|_{L^{\frac 5 3 }_x L^{20}_t}\lesssim 2^{\frac j
  {10}}.
$$
Suming over low frequencies recovers the desired bound. Notice that
scaling dictates the exponents (hence there is no need to compute
explicitely the interpolation $\theta$).

\subsection{Scattering for $3\leq p\leq 3+2/5$}
In this part we consider the remaining case, e.g. nonlinearities which
are close to $3$ and for which our main results do not provide a
scale-invariant local Cauchy theory. As mentioned before, this case
will be dealt with using the approach from \cite{plve08}. As such,
this entire Subsection is somewhat disconnected from the rest of the
paper; the combination of several technical difficulties makes it
lenghty and cumbersome, but we hope the underlying strategy is
clear. We have two a priori bounds on the nonlinear equation at our
disposal: local smoothing, which is at the scale of $\dot H^\frac 1 2$
regularity for the data, and an $L^4_{t,x}$ space-time bound, which is
at the scale of $\dot H^\frac 1 4$ regularity for the data. Both are
below the scale of critical $H^{s}$ regularity, which is $s_p=\frac 3
2-\frac 2 {(p-1)}$. Interpolation with the energy bound provides bounds at the
critical level, but the lack of flexible scale-invariant estimates on
the inhomogeneous problem make them seemingly useless. As such, one
has to improve both the local smoothing bound and the $L^4_{t,x}$
space-time bounds obtained in \cite{plve08}, to reach critical scaling and beyond. This is
accomplished through several steps, which we informally summarize as follows:
\begin{itemize}
\item improve the space-time bounds by using the equations far and
  close to the boundary. As the resulting commutator source term can
  only be handle at $H^{\frac 1 2}$ regularity, this will improve
  estimates from $\dot H^\frac 1 4$ regularity to $\dot H^{\frac 1
    2-\varepsilon}$ regularity, which is still below scale invariance;
\item combine this improved estimates with the energy bound to obtain yet again better space-time bounds through the equation (but
  splitting the source terms in close and far away terms). As an added
  bonus we also improve our local smoothing estimate; moreover we now
  go beyond scale-invariance;
\item turn the crank a few more times, going back and forth between
  estimates on the split equations and estimates on the equation with
  split source terms, until we reach the correct set of estimates to
  prove scattering at the $H^1_0$ regularity. It is worth noticing
  that the numerology gets worse with $p>3+2/5$, and that the
  forthcoming argument would probably break down before even reaching $p=4$.
\end{itemize}
We start by stating a few linear estimates which will be needed in the
proof and are simple consequences of our Theorem \ref{thmestiminbis}
by summing over dyadic frequencies.
\begin{lemma}\label{lem54}(see \cite[Lemma 5.4]{plve08})
Let $\Omega$ be a non trapping domain and denote $u_{L}=e^{it\Delta_{D}}$ the linear flow for the Schr\"odinger equation on $\Omega$ with Dirichlet boundary conditions. Then
\begin{equation}\label{}
\|e^{it\Delta_{D}}u_{0}\|_{L^{4}_{t}\dot{W}^{s,4}(\Omega)}\lesssim
\|u_{0}\|_{\dot{H}^{s+\frac 1 4}_{0}(\Omega)}.
\end{equation}
Denote by $w$ the solution of the inhomogeneous equation, e.g. $w=\int_{0}^{t}e^{i(t-s)\Delta_{D}}f(s)ds$, then
\begin{equation}\label{sertplusloin}
\|w\|_{C_{t}\dot{H}^{s+\frac 1
    4}_{0}(\Omega)}+\|w\|_{L^{4}_{t}\dot{W}^{s,4}}\lesssim
\|f\|_{L^{\frac 4 3}_{t}\dot{W}^{s+\frac 1 2,\frac 4 3}}.
\end{equation}
\end{lemma}
The next lemma is just the Christ-Kiselev lemma again, stated in a form
which is convenient for later use.
\begin{lemma}\label{lem56}(see \cite[Lemma 5.6]{plve08})
Let $U(t)$ be a one parameter group of operators, $1\leq r<q\leq \infty$, $H$ an Hilbert space and $B_{r}$ and $B_{q}$ two Banach spaces. Suppose that
\[
\|U(t)\varphi\|_{L^{q}_{t}(B_{q})}\lesssim \|\varphi\|_{H} \quad \text{and}\quad \|\int_{s}U(-s)g(s) ds\|_{H}\lesssim \|g\|_{L^{r}_{t}(B_{r})},
\]
then
\[
\|\int_{s<t} U(t-s)g(s)ds\|_{L^{q}_{t}(B_{q})}\lesssim \|g\|_{L^{r}_{t}(B_{r})}.
\]
\end{lemma}
finally, we recall that we have Lemma \ref{ST} at our disposal, should
we need the endpoint Strichartz on the left handside in Lemma
\ref{lem56}, provided that we used a (dual) local smoothing norm on
the right handside.

In what follows we shall write $p=3+2\eta$, with $\eta\in
[0,1/5]$. All the nonlinear mappings which we use can be proved using
the appendix and we will no longer refer to it. We recall all a priori
bounds at our disposal: the first two are uniform in time bounds for
the $L^2(\Omega)$ and $H^1_0(\Omega)$ norms of the solution to the
defocusing NLS, irrespective of the power $p$, and were already stated in the
previous section, see \ref{massenerg}. The next two were obtained in
\cite{plve08}, again in the defocusing case and irrespective of $p$: a
space-time norm estimate
\begin{equation}
  \label{aprioriL4}
  \|u\|_{L^4_t(L^4(\Omega))} \leq E^{\frac 1 8} M^\frac 3 8,
\end{equation}
which has the same scaling as $\dot H^\frac 1 4$ for the data; and a
local smoothing norm estimate
\begin{equation}
  \label{aprioriLS}
  \|\nabla u\|_{L^2_t(L^2(K))} \leq C(K) E^{\frac 1 4} M^\frac 1 4,
\end{equation}
which has the same scaling as $\dot H^\frac 1 2$ for the data; here
$K$ is meant to be a compact set which includes the obstacle, and
\eqref{aprioriLS} holds only under the star-shaped condition on the
obstacle, while proving  \eqref{aprioriL4} makes an essential use of \eqref{aprioriLS}.

 We start with proving
\begin{prop}\label{propet1}
Let $u$ be a solution to the nonlinear problem \eqref{eqp}. Let $\chi\in C^{2}_{0}(\mathbb{R}^{3})$  be a smooth function equal to $1$ near $\partial\Omega$. Then 
\begin{equation}\label{nlset1}
\chi u\in L^{4}_{t}\dot{B}^{1/4-\eta,2}_{4}(\Omega)\quad \text{and}\quad (1-\chi)u\in L^{2}_{t}\dot{B}^{1/2-\eta,2}_{6}(\Omega).
\end{equation}
\end{prop}
\begin{rmq}
  Notice that our cut $\chi$ is only $C^2$ rather thant $C^\infty$,
  and this will remain so for the rest of the section. This is in no
  way a difficulty, and it allows to conveniently take $\chi=\chi_1^p$
  or $\chi=\chi_1^{p-1}$, where $\chi_1\in C^2_0$ as
  an admissible cut if we need, as $p-1>2$. This is particulary
  convenient for nonlinear mappings where all factors can be
  considered \og equal\fg{}. Alternatively, one may retain
  $C^\infty_0$ cuts and play with at least 3 overlapping ones, as was
  done in \cite{plve08}, at the expense of desymetrizing various
  nonlinear estimates. These are (mildly ennoying) considerations
  that the reader should ignore at first read.
\end{rmq}
\begin{proof}
In order to prove the Proposition, we split the equation \eqref{eqp}, treating differently the neighborhood of the boundary (using local smoothing type arguments) and spatial infinity (where the full range of sharp Stricharz estimates holds). 

Consider the equation satisfied by $\chi u$,
\begin{equation}\label{eqnlchiu}
(i\partial_{t}+\Delta_{D})(\chi u)= \chi |u|^{2+2\eta}u-[\chi,\Delta_{D}]u.
\end{equation}
We need to show that the nonlinear term belongs to
$L^{2}_{t}H^{-\eta}_{comp}(\Omega)$. The commutator term is controlled by
$\|\tilde{\chi}u\|_{L^{2}_{t}H^{1}_{comp}}$ for some $\tilde{\chi}\in
C^{2}_{0}(\mathbb{R}^{3})$ equal to $1$ on the support of $\chi$
and it belongs to $L^{2}_{t}L^{2}_{comp}(\Omega)\subset
L^{2}_{t}H^{-\eta}_{comp}(\Omega)$. We now deal with the nonlinear term: let $q$ be such that
$\dot{B}^{1,2}_{q}(\Omega)\subset H^{-\eta}(\Omega)$, hence
$1-\frac{3}{q}=-\eta-\frac{3}{2}$. Then
$\frac{1}{q}=\frac{1}{2}+\frac{2(1+\eta)}{6}$ and
\begin{equation*}
\|\chi |u|^{2(1+\eta)}u\|_{L^{2}_{t}H^{-\eta}_{comp0}(\Omega)}\lesssim
\|\chi |u|^{2(1+\eta)}u\|_{L^{2}_{t}\dot{B}^{1,2}_{q}(\Omega)}\lesssim
\|\chi_1
u\|_{L^{2}_{t}H^{1}_{0}(\Omega)}\|(\chi_1 u)^{1+\eta}\|_{L^{\infty}_{t}L^{\frac{6}{1+\eta}}(\Omega)},
\end{equation*}
where $\chi_1^p=\chi$ and we used $u\in
L^{\infty}_{t}H^{1}_{0}(\Omega)\subset L^{\infty}_{t}L^{6}(\Omega)$ on
two factors and $u\in L^{2}_{t}H^{1}_{comp}(\Omega)$ on one factor.
Hence  the right hand side in \eqref{eqnlchiu} is in
$L^{2}_{t}H^{-\eta}_{comp}(\Omega)$ and we can apply Lemma \ref{lem56} with
$L^{q}(B_{q}):=L^{4}_{t}\dot{W}^{1/4-\eta,4}(\Omega)$,
$H:=H^{1/2-\eta}(\Omega)$ and
$L^{r}(B_{r}):=L^{2}_{t}H^{-\eta}_{comp}(\Omega)$. This gives the
first assertion in \eqref{nlset1}. 
Let us deal now with $(1-\chi)u$ which is solution to
\begin{equation}\label{eqnlhorsbd}
(i\partial_{t}+\Delta_{D})((1-\chi) u)= (1-\chi) |u|^{2+2\eta}u+[\chi,\Delta]u,
\end{equation}
where $\Delta$ denotes the free Laplacian (notice that we can consider
\eqref{eqnlhorsbd} in the whole space $\mathbb{R}^{3}$ since both
source terms vanish near the boundary $\partial\Omega$). The commutator term
is dealt with exactly as in the previous part and is therefore in $L^{2}_{t}L^{2}_{comp}(\Omega)$.

Let $v:=(1-\chi_{1})u$ for some $\chi_{1}\in C^2_{0}(\mathbb{R}^{3})$
such that $(1-\chi_{1})^{p}=1-\chi$.
In order to prove \eqref{nlset1} we only need to prove
$|v|^{2+2\eta}v\in L^{2}_{t}\dot{B}^{1/2-\eta,2}_{6/5}(\Omega)$, since
then we may apply the dual end-point Strichartz estimates (from the
$\R^3$ case) on the
nonlinear term. Using the embedding
$\dot{B}^{1-\eta,2}_{1}(\Omega)\subset\dot{B}^{1/2-\eta,2}_{6/5}(\Omega)$,
it suffices to get $|v|^{2+2\eta}v\in
L^{2}_{t}\dot{B}^{1-\eta,2}_{1}(\Omega)$. When evaluating the \og
product\fg{} $|v|^{2+2\eta}v$ we will use  for one factor $v$ the
energy bound and Sobolev embedding,  $L^{\infty}_{t}H^{1}_{0}(\Omega)\subset
L^{\infty}_{t}\dot{B}^{1-\eta,2}_{q}(\Omega)$ with
$\frac{1}{q}=\frac{1}{2}-\frac{\eta}{3}$. On the other hand, from our
a priori bound from \cite{plve08}, we have $v\in
L^{4}_{t}L^{4}(\Omega)$, while $v\in L^{\infty}_{t}H^{1}_{0}(\Omega)\subset
L^{\infty}_{t}L^{6}(\Omega)$ and hence $v^{1+\eta}\in
L^{4/(1+\eta)}_{t}L^{4/(1+\eta)}(\Omega)\cap
L^{\infty}_{t}L^{6/(1+\eta)}(\Omega)$. Interpolation with weights
$1/(1+\eta)$ and $\eta/(1+\eta)$ gives $v^{1+\eta}\in
L^{4}_{t}L^{12/(3+2\eta)}(\Omega)$. Consequently,
\begin{equation*}
\||v|^{2+2\eta}v\|_{L^{2}_{t}\dot{B}^{1/2-\eta,2}_{6/5}(\Omega)}\lesssim \||v|^{2+2\eta}v\|_{L^{2}_{t}\dot{B}^{1-\eta,2}_{1}(\Omega)}
\lesssim
\|v\|_{L^{\infty}_{t}\dot{B}^{1-\eta,2}_{q}(\Omega)}\||v|^{1+\eta}\|^{2}_{L^{4}_{t}L^{12/(3+2\eta))}(\Omega)}.
\end{equation*}
This achieves the proof of Proposition \ref{propet1}.
\end{proof}
\begin{rmq}
  One should point out that the proof of this last estimate is
  slightly incorrect, as it conveniently ignores the situation where
  low frequencies are on the $v$ factor and high frequencies are on
  $|v|^{2+2\eta}$. This can be easily fixed by revisiting the proof of
  Lemma \ref{para} and \ref{diffpara} in the Appendix, noticing that
  we may suppose that factors $f$ there are in several different
  $L^r$ spaces and distribute them when using H\"older on the low
  frequencies in the proofs. The same situation occurs several times
  in the present proof and we leave details to the reader.
\end{rmq}
The next iterative step will be the following lemma: 
\begin{prop}\label{propet2}
Let $u$ be a solution to the nonlinear problem \eqref{eqp}. Then 
\begin{equation}\label{utoprove}
u\in L^{4}_{t}\dot{W}^{1/4+\eta,4}(\Omega)\cap L^{2}_{t}H^{1+\eta}_{comp}(\Omega).
\end{equation}
\end{prop}
\begin{proof}
The split of the equation into equations for $\chi u$
and $(1-\chi)u$ is no longer of any use: the resulting
commutator source term is no better than
$[\chi,\Delta]u\in L^{2}_{t}L^{2}_{comp}(\Omega)$. However we now have
estimates from Proposition \ref{propet1} which turn out to be good enough that splitting the nonlinear term in \eqref{eqp} in two parts, using
the partition $\chi+(1-\chi)=1$ will allow us to use the somewhat
restricted set of inhomogeneous estimates we have for the equation on
a domain. Setting $g_{1}:=\chi |u|^{2+2\eta}u$,
$g_{2}:=(1-\chi)|u|^{2+2\eta}u$ and using Duhamel formula, we have
\begin{equation}\label{duhnlu}
u(t,x)=e^{it\Delta_{D}}u_{0}+\int_{0}^{t}e^{i(t-s)\Delta_{D}}g_{1}(s)ds+\int_{0}^{t}e^{i(t-s)\Delta_{D}}g_{2}(s)ds\,;
\end{equation}
the idea is then that one may use \eqref{sertplusloin} on the $g_1$
Duhamel term, while the $g_2$ term may be handled in $L^1_t(\dot H^s)$
for a suitable $s$.
\begin{lemma}\label{lemet12}
Let $v:=(1-\chi_{1})u$, where $\chi_{1}\in C^2_{0}(\mathbb{R}^{3})$ is
such that $(1-\chi_{1})^{p}=1-\chi$. We have
\begin{equation}\label{nlg2}
g_{2}\in L^{2}_{t}\dot{B}^{1/2,2}_{6/5}(\Omega)\quad \text{and}\quad  v\in L^{2}_{t}\dot{B}^{1/2,2}_{6}.
\end{equation} 
Moreover, $g_2\in L^1_t(\dot H^{\frac 1 2 +\eta}(\Omega))$ and 
\begin{equation}
  \label{eq:jenaimarre}
\|\int_{0}^{t}e^{i(t-s)\Delta_{D}}g_{2}(s)ds\|_{L^{4}_{t}\dot{B}^{1/4+\eta,2}_{4}(\Omega)\cap
  L^{2}_{t}H^{1+\eta}_{comp}(\Omega)} \lesssim \| g_2\|_{L^1_t(\dot
  H^{\frac 1 2+\eta}(\Omega))}.  
\end{equation}
\end{lemma}
\begin{proof}
From Proposition \ref{propet1}, the energy
and mass bound, and interpolation, we have 
\[
v\in L^{2}_{t}\dot{W}^{1/2-\eta,6}(\Omega)\cap L^\infty_t(\dot
H^{\frac 1 2-\eta}(\Omega)\subset L^{4}_{t}L^{q}(\Omega)\quad \text{for} \quad \frac{1}{q}=\frac{1}{6}+\frac{\eta}{3},
\]
hence $|v|^{1+\eta}\in L^{4/(1+\eta)}_{t}L^{q/(1+\eta)}(\Omega)\cap
L^{\infty}_{t}L^{6/(1+\eta)}(\Omega)$. We now interpolate again and obtain
$|v|^{1+\eta}\in L^{4}_{t}L^{r}(\Omega)$, where
$\frac{2}{r}=\frac{1}{3}+\eta$. Therefore, the nonlinear term
$g_{2}=|v|^{2+2\eta}v$ belongs to
$L^{2}_{t}\dot{B}^{1-3\eta,2}_{6/5}(\Omega)$. Indeed, let
$\frac{1}{m}=\frac{1}{2}+\frac{2}{r}=\frac{5}{6}+\eta$, then
\begin{equation}
\|g_{2}\|_{L^{2}_{t}\dot{B}^{1-3\eta,2}_{6/5}(\Omega)}\lesssim \|g_{2}\|_{L^{2}_{t}\dot{B}^{1,2}_{m}(\Omega)}\lesssim\|v\|_{L^{\infty}_{t}\dot{H}^{1}_{0}(\Omega)}\||v|^{1+\eta}\|^{2}_{L^{4}_{t}L^{r}(\Omega)}.
\end{equation}
If $1-3\eta\geq 1/2$, \eqref{nlg2} follows, but unfortunately this covers
only $\eta\leq 1/6$. It remains to deal with the situation $\eta\in
(1/6,1/5]$. In this case we use the equation satisfied by $v$
(obtained by replacing $\chi$ by $\chi_{1}$ in \eqref{eqnlhorsbd}) to
get
\begin{equation}
v\in L^{2}_{t}\dot{B}^{1-3\eta,2}_{6}(\Omega).
\end{equation}
In fact, the commutator term $[\chi_{1},\Delta]u$ is in
$L^{2}_{t}L^{2}(\Omega)$ and, consequently, it also belongs to
$L^{2}_{t}H^{1/2-3\eta}(\Omega)$ since in this case $1/2-3\eta<0$,
while $(1-\chi_{1})|v|^{2+2\eta}v\in
L^{2}_{t}\dot{B}^{1-3\eta,2}_{6/5}(\Omega)$ as shown
before. Therefore, with $1-3\eta-3/r= 2(1-3\eta)-1$,
\begin{equation}\label{inclv1}
v|v|\in L^{1}_{t}\dot{B}^{1-3\eta,2}_{r}(\Omega)\subset L^{1}_{t}\dot{B}^{1-6\eta,2}_{\infty}(\Omega).
\end{equation}
In order to estimate $g_{2}$ we use \eqref{inclv1} for a factor $v|v|$, while for the remaining factor $|v|^{1+2\eta}$ we use $v\in L^{\infty}_{t}H^{1}_{0}(\Omega)$, which yields 
\begin{equation}\label{inclv2}
|v|^{1+2\eta}\subset
L^{\infty}_{t}\dot{B}^{1,2}_{\lambda}(\Omega)\subset
L^{\infty}_{t}H^{1-\eta}(\Omega)\quad \text{for}\quad
\frac{1}{\lambda}=\frac{1}{2}+\frac{\eta}{3}.
\end{equation}
From \eqref{inclv1}, \eqref{inclv2} and product rules, we get
$g_{2}\in L^{1}_{t}H^{2-7\eta}(\Omega)\subset
L^{1}_{t}H^{1/2}(\Omega)$ (notice that the regularity is
$1-\eta-(6\eta-1)$ where $6\eta-1>0$). 

Using the equation satisfied by $v$ and Duhamel formula we can write
\begin{equation}\label{inclv0}
v(t,x)=e^{it\Delta_{\R^3}}(1-\chi_{1})u_{0}+\int_{0}^{t}e^{i(t-s)\Delta_{\R^3}}(g_{2}+[\chi_{1},\Delta]u)(s)ds.
\end{equation}
Using Lemma \ref{lem54} with
$L^{q}(B_{q}):=L^{2}_{t}\dot{B}^{1/2,2}_{6}(\Omega)$,
$L^{r}(B_{r}):=L^{1}_{t}H^{1/2}(\Omega)$, the first term in the
integral in the right hand side of \eqref{inclv0} belongs to
$L^{2}_{t}\dot{B}^{1/2,2}_{6}(\Omega)$. Using Lemma \ref{ST}, we also obtain
\begin{equation*}
\|\int_{0}^{t}e^{i(t-s)\Delta}[\chi_{1},\Delta]u(s)ds\|_{L^{2}_{t}\dot{B}^{1/2,2}_{6}(\Omega)}
\lesssim
\|[\chi_{1},\Delta]u\|_{L^{2}_{t}L^{2}_{comp}(\Omega)}.
\end{equation*}
Finally, the linear evolution $e^{it\Delta_{\R^3}}(1-\chi_{1})u_{0}$
is evidently in $L^{2}_{t}\dot B^{1/2,2}_6(\Omega)$ and we obtain \eqref{nlg2}.
\begin{rmq}
For the last part of the proof of Lemma \ref{lemet12} we shall use less information than that, precisely we only need the fact that for $\epsilon>0$ small enough we have 
\begin{equation}\label{inclv3}
v\in L^{2}_{t}\dot{B}^{1/2-\epsilon,2}_{6}(\Omega)\subset L^2_t (L^\frac 3 \epsilon(\Omega)) \subset L^{2}_{t}\dot{B}^{-\epsilon,\infty}_{\infty}(\Omega),
\end{equation}
and $|v|\in L^\frac 3 \epsilon(\Omega)\subset L^{2}_{t}\dot{B}^{-\epsilon,\infty}_{\infty}(\Omega)$ as well.
\end{rmq}
We refine our knowledge on $g_{2}=v|v| v^{1+2\eta}$: using the
previous remark, we now have $v|v|\in
L^{1}_{t}\dot{B}^{-2\epsilon,\infty}_{\infty}(\Omega)$. From \eqref{inclv2}
we also have $|v|^{1+2\eta}\in
L^{\infty}_{t}\dot{B}^{1,2}_{\lambda}(\Omega)$ if $\lambda
=\frac{6}{3+2\eta}$. Thus, the source term $g_{2}$ can be estimated as
follows 
\begin{equation}\label{estutilg2}
\|g_{2}\|_{L^{1}_{t}H^{1-\eta-2\epsilon}(\Omega)}\lesssim\|g_{2}\|_{L^{1}_{t}\dot{B}^{1-2\epsilon,2}_{\lambda}(\Omega)}\lesssim
\|v|v|\|_{L^{1}_{t}\dot{B}^{-2\epsilon,\infty}_{\infty}(\Omega)}\||v|^{1+2\eta}\|_{L^{\infty}_{t}\dot{B}^{1,2}_{\lambda}(\Omega)}.
\end{equation}
Using again Lemma \ref{lem54}, this time with
$L^{q}(B_{q}):=L^{4}_{t}\dot{B}^{3/4-\eta-2\epsilon,2}_{4}(\Omega)$,
$H:=H^{1-\eta-2\epsilon}(\Omega)$ and
$L^{r}(B_{r}):=L^{1}_{t}H^{1-\eta-2\epsilon}(\Omega)$, we get by
interpolation
\begin{multline}\label{inclv5}
\|\int_{0}^{t}e^{i(t-s)\Delta}g_{2}(s)ds\|_{L^{4}_{t}\dot{B}^{1/4+\eta,2}_{4}(\Omega)}\lesssim
\|\int_{0}^{t}e^{i(t-s)\Delta}g_{2}(s)ds\|^\theta_{L^{4}_{t}{B}^{3/4-\eta-2\epsilon,2}_{4}(\Omega)}\|u\|^{1-\theta}_{L^4_{t,x}}\\
\lesssim \|g_{2}\|_{L^{1}_{t}H^{1-\eta-2\epsilon}(\Omega)}+\|u\|_{L^4_{t,x}}\,;
\end{multline}
where for the first (interpolation) inequality in \eqref{inclv5} we used that $3/4-\eta-2\epsilon>1/4+\eta$ if $\epsilon$ is sufficiently small (take $0<\epsilon\leq 1/20$ for example). 

On the other hand, by Lemma \ref{lem56} again,
\begin{equation}\label{inclv6}
 \|\int_{0}^{t}e^{i(t-s)\Delta}g_{2}(s)ds\|_{L^{2}_{t}H^{1+\eta}_{comp}(\Omega)} \lesssim \|g_{2}\|_{L^{1}_{t}H^{1/2+\eta}(\Omega)}\lesssim \|g_{2}\|_{L^{1}_{t}H^{1-\eta-2\epsilon}(\Omega)},
\end{equation}
which finally achieves the proof of Lemma \ref{lemet12}.
\end{proof}
It remains now to deal with the Duhamel term coming from $g_1$ in
\eqref{duhnlu}.
\begin{lemma}\label{lemet1}
Suppose that we know moreover that 
\begin{equation}\label{assumpv}
u\in L^{4}_{t}\dot{B}^{\sigma,2}_{4}(\Omega),\quad \text{where} \quad \sigma=\frac{1}{4}+\frac{\eta}{1+\eta},
\end{equation} 
then 
\begin{equation}\label{nlg1}
g_{1}\in L^{4/3}_{t}\dot{B}^{3/4+\eta}_{4/3}(\Omega)\quad \text{and}\quad \int_{0}^{t}e^{i(t-s)\Delta_{D}}g_{1}(s)ds\in L^{4}_{t}\dot{B}^{1/4+\eta,2}_{4}\cap L^{2}_{t}H^{1+\eta}_{comp}(\Omega).
\end{equation} 
\end{lemma}
Taking the lemma for granted, we can complete the proof of Proposition
\ref{propet2}: using Lemmas \ref{lemet12}, \ref{lemet1}, the fact that
the linear flow is in $L^{\infty}_{t}H^{1}_0(\Omega)\cap
L^{2}_{t}H^{3/2}_{comp}(\Omega)$
 and Duhamel formula \eqref{duhnlu}, estimate \eqref{utoprove} follows immediately.
\begin{proof}\emph{(of Lemma \ref{lemet1}):}
The a-priori information \eqref{assumpv} gives
\[
u\in L^{4}_{t}\dot{B}^{\sigma,2}_{4}(\Omega)\subset L^{4}_{t}L^{q}(\Omega)\quad \text{for}\quad  \frac{1}{q}=\frac{1}{4}-\frac{\sigma}{3},
\]
and consequently $u^{2(1+\eta)}\in
L^{2/(1+\eta)}_{t}L^{3/(1-\eta)}(\Omega)$. On the other hand,
interpolating between $L^{2}_{t}H^{1}_{comp}(\Omega)$ and
$L^{\infty}_{t}H^{1}_{0}(\Omega)$ gives $\chi u\in
L^{r}_{t}H^{1}_{comp}(\Omega)$ for every $r\in [2,\infty]$. Therefore,
with $\chi_1^p=\chi$, we
can estimate
\begin{equation}
\|\chi |u|^{2+2\eta}u\|_{L^{4/3}_{t}\dot{B}^{1,2}_{M}}\lesssim \|\chi_1 u\|_{L^{4/(1-2\eta)}_{t}H^{1}_{comp}(\Omega)}\|u^{2+2\eta}\|_{L^{2/(1+\eta)}_{t}L^{3/(1-\eta)}(\Omega)},
\end{equation}
where
$\frac{1}{M}=\frac{1}{2}+\frac{1-\eta}{3}=\frac{5}{6}-\frac{\eta}{3}$. It
remains to notice that for $M$ defined above, the embedding
$\dot{B}^{1,2}_{M}(\Omega)\subset\dot{B}^{3/4+\eta,2}_{4/3}(\Omega)$
holds (indeed, $1>3/4+\eta$ and $1-3/M=3/4+\eta-9/4$) and to use again
Lemmas \ref{lem56}, \ref{ST}. Another application of Lemma
\ref{lem56} with $L^{q}(B_{q}):=L^{2}_{t}H^{1+\eta}_{comp}(\Omega)$,
$H:=H^{1/2+\eta}_{comp}(\Omega)$ and
$L^{r}(B_{r}):=L^{4/3}_{t}\dot{B}^{3/4+\eta,2}_{4/3}(\Omega)$ achieves
the proof  of \eqref{nlg1} and Lemma \ref{lemet1}.
\end{proof}
\emph{End of the proof of Proposition \ref{propet2}:}
In order to complete the proof of Proposition \ref{propet2} it remains
to prove that \eqref{assumpv} holds indeed, since we have used it to
deduce \eqref{utoprove}. Let $0<T<\infty$ be small enough, so that
 by the local existence theory (see \cite{plve08}) the
 $L^{4}_{T}\dot{B}^{\sigma,2}_{4}(\Omega)$ norm of $u$ is finite; in fact,
 the same can be said with $\sigma$ replaced by $\eta+\frac 1 4$. We
 shall prove that $T=\infty$ is allowed. For this, we interpolate between
$L^{4}_{t}\dot{B}^{1/4-\eta,2}_{4}(\Omega)$ and
$L^{4}_{T}\dot{B}^{1/4+\eta,2}_{4}(\Omega)$ with interpolation exponent $\theta=\frac{\eta}{2(1+\eta)}$
to obtain an estimate on the $L^{4}_{T}\dot{B}^{\sigma,2}_{4}(\Omega)$
norm, where $\sigma=1/4+\eta/(1+\eta)$:
\begin{equation}\label{es1}
\|u\|_{L^{4}_{T}\dot{B}^{\sigma,2}_{4}(\Omega)}\leq \|u\|^{\theta}_{L^{4}_{t}\dot{B}^{1/4-\eta,2}_{4}(\Omega)}\|u\|^{1-\theta}_{L^{4}_{T}\dot{B}^{1/4+\eta,2}_{4}(\Omega)}.
\end{equation}
Recall that from Proposition \ref{propet1} we have now a uniform bound,
\begin{equation}\label{es2}
\|u\|_{L^{4}_{t}\dot{B}^{1/4-\eta,2}_{4}(\Omega)}\lesssim C(E,M),
\end{equation}
and from Lemma \ref{lemet12} we consequently also have a uniform bound on the
Duhamel part coming from $g_2$, see \eqref{eq:jenaimarre}.
Finally, using  \eqref{nlg1} for $g_1$ and the uniform bounds we
already have for the linear part and the $g_2$ part,
\begin{equation}\label{es3}
\|u\|_{L^{4}_{T}\dot{B}^{1/4+\eta,2}_{4}(\Omega)}\lesssim
C_1(E,M)+C_2(E,M)\|\chi u\|^{1/2-\eta}_{L^{2}_{t}H^{1}_{comp}(\Omega)}
\|u\|^{2(1+\eta)}_{L^{4}_{T}\dot{B}^{\sigma,2}_{4}(\Omega)}.
\end{equation}
Plugging \eqref{es2}, \eqref{es3} in \eqref{es1} yields
\begin{equation}
\|u\|_{L^{4}_{T}\dot{B}^{\sigma,2}_{4}(\Omega)}\leq C_{3}(E,M)+C_{4}(E,M)\|\chi u\|^{\gamma}_{L^{2}_{t}H^{1}_{comp}(\Omega)}\|u\|^{\rho}_{L^{4}_{T}\dot{B}^{\sigma,2}_{4}(\Omega)},
\end{equation}
where $\rho,\gamma>0$. The coefficients are uniformly bounded, and a
splitting time argument performed on the
$L^{2}_{t}H^{1}_{comp}(\Omega)$ norm which is finite provides global
in time control of $u$ in $L^{4}_{t}\dot{B}^{\sigma,2}_{4}(\Omega)$. This
finally completes the proof of Proposition \ref{propet2}.
\end{proof}
\begin{rmq}
  The space $L^4_t(\dot B^{\sigma,2}_4(\Omega))$ with $\sigma=\frac 1
  4+\frac{\eta}{1+\eta}$ does not show up by accident: rather, it is a
  scale invariant space with respect to the critical regularity
  $s_p$. As such, it makes sense that it plays a pivotal role in the
  argument. Having reached (and in fact, gone beyond) critical scaling
  in our a priori estimates, the remaining part of the argument is
  somewhat less involved.
\end{rmq}
At this point of the proof, we could establish scattering in the
scale-invariant Sobolev space; however we want to reach
$H^1_0$. Recall that we may write
\begin{equation*}
\|
u(t,x)-e^{it\Delta_{D}}(u_{0}+\int_{0}^{+\infty}e^{-is\Delta_{D}}|u|^{p-1}u(s)ds)\|_{H^1_0}
= \| \int_{t}^{+\infty}e^{i(t-s)\Delta_{D}}|u|^{p-1}u(s)ds\|_{H^1_0},
\end{equation*}
from which we wish to use Duhamel to get
\begin{equation}
  \label{scatter}
\|
\int_{t}^{+\infty}e^{i(t-s)\Delta_{D}}|u|^{p-1}u(s)ds\|_{H^1_0}\lesssim
\|g_1\|_{L^{4/3}(t,+\infty;\dot B^{5/4,2}_{4/3}(\Omega))}+\|g_2\|_{L^{1}(t,+\infty;H^1_0(\Omega))},
\end{equation}
from which scattering easily follows (the same argument applies at
$t=-\infty$ as well). 

Therefore we focus on the right handside and start with the easiest
part, which is $g_2$.
\begin{lemma}\label{lemet3g2}
We have $g_{2}=(1-\chi)u^{p}\in L^{1}_{t}H^{1}_{0}(\Omega)$. 
\end{lemma}
\begin{proof}
We start by proving that
\begin{equation}\label{scatv}
v=(1-\chi_{1})u\in L^{2(1+\eta)}_{t}L^{\infty}(\Omega).
\end{equation}
\begin{rmq}
Notice that if we have \eqref{scatv} the proof is finished since then
\begin{equation}
\|v|v|^{2+2\eta}\|_{L^{1}_{t}H^{1}_{0}(\Omega)}\leq \||v|^{2(1+\eta)}\|_{L^{1}_{t}L^{\infty}(\Omega)}\|v\|_{L^{\infty}_{t}H^{1}_{0}(\Omega)}.
\end{equation}
\end{rmq}
We proceed with \eqref{scatv}. From Lemma \ref{lemet12} we know that
$g_{2}\in L^{1}_{t}H^{1-\eta}(\Omega)$ and $[\chi,\Delta_{D}]u\in
L^{2}_{t}H^{\eta}_{comp}(\Omega)$, so using again the equation for $(1-\chi)u$ and Lemma \ref{lem56},
\begin{equation}
(1-\chi)u\in L^{2}_{t}\dot{B}^{1-\eta,2}_{6}(\Omega)\bigl (\cap
L^{\infty}_{t}H^{1}_{0}(\Omega)\bigr ).
\end{equation}
Recall that from Lemma \ref{lemet12} we also have $v\in
L^{2}_{t}\dot{B}^{1/2,2}_{6}\cap L^{\infty}_{t}H^{1/2}(\Omega)$. The
Lemma now follows by interpolation and the Gagliardo-Nirenberg
inequality (a similar key step exists in \cite{plve08}).
\end{proof}
\begin{lemma}\label{lemet3g1}
We have $g_{1}=\chi u^{p}\in L^{4/3}_{t}\dot{B}^{5/4,2}_{4/3}(\Omega)$.
\end{lemma}
\begin{proof}
We first prove
\begin{equation}\label{l8eta}
u\in L^{8(1+\eta)}_{t}L^{8(1+\eta)}(\Omega).
\end{equation}
Indeed, from Propositions \ref{propet1}, \ref{propet2} and
interpolation, we get $u\in
L^{4}_{t}\dot{B}^{1/4+\eta/2,2}_{4}(\Omega)$. Interpolating again between
this bound and the energy bound $u\in L^{\infty}_{t}H^{1}_{0}(\Omega)$,
followed by Sobolev embedding yields \eqref{l8eta}.
Now we write
\begin{equation}
\|g_{1}\|_{L^{4/3}_{t}\dot{B}^{5/4,2}_{4/3}(\Omega)}\lesssim \|\chi u\|_{L^{2}_{t}H^{5/4}_{comp}(\Omega)}\|u^{2+2\eta}\|_{L^{4}_{t}L^{4}(\Omega)},
\end{equation}
and also by the Duhamel formula and the local smoothing estimate on
the domain,
\begin{equation}
\|u\|_{L^{2}_{t}H^{5/4}_{comp}(\Omega)}\leq \|u_{0}\|_{H^{3/4}(\Omega)}+\|g_{1}\|_{L^{4/3}_{t}\dot{B}^{1,2}_{4/3}(\Omega)}+\|g_{2}\|_{L^{1}_{t}H^{3/4}(\Omega)}.
\end{equation}
Certainly, using Lemma \ref{lemet3g2}, the $g_2$ term is bounded. For
$g_1$, we may write
\begin{equation}
\|g_{1}\|_{L^{4/3}_{t}\dot{B}^{1,2}_{4/3}(\Omega)}\lesssim  \|\chi
u\|_{L^{2}_{t}H^{1}_{comp}(\Omega)}\|u^{2+2\eta}\|_{L^{4}_{t}L^{4}(\Omega)};
\end{equation}
and we have reached a point where our right handside is uniformly
bounded. Consequently the Lemma is proved, and  this concludes the
proof of Theorem \ref{thmleqscat}.
\end{proof}

\section*{Appendix}

In order to perform the various product estimates, we need a
couple of useful lemma. Observe that with the spectral localization
one cannot take advantage of convolution of Fourier supports. As a
first step and in order to avoid cumbersome notations, we only
consider functions and Besov spaces which do not depend on time. We
will then explain how to re-instate the time dependance in the
nonlinear estimates.

It is worth noting at this stage, however, that both $\Delta_j$ and
$S_j$ operators are well-defined on $L^p_t L^q_x$ and $L^q_x L^p_t$
for all the pairs $(p,q)$ to be considered: this follows from
\cite{ip08} for the case $L^p_t L^q_x$ where the time norm is
harmless. In the case $L^q_x L^2_t$, the arguments from \cite{ip08}
apply as well (heat estimates are proved for data in $L^p_x(H)$ where
$H$ is an abstract Hilbert space, and when $H=L^2_t$, the heat kernel
is diagonal and therefore Gaussian as well). By interpolation and
duality we recover all pairs $(p,q)$.
\begin{rmq}
  In $\R^n$, one may perform product estimates in an easier way
  because of the convolution of Fourier supports. However, when
  dealing with non integer power-like nonlinearities, one cannot
  proceed so easily: the usual route is to use a characterization of
  Besov spaces via finite differences; here, because of the Banach
  valued Besov spaces, we perform a direct argument which is directly
  inspired by computations in \cite{pl02}, where the same sort of
  time-valued Besov spaces were unavoidable.
\end{rmq}
\begin{lemma}
\label{paraHF}
  Let $f_j$ be such that $S_j f_j=f_j$, and $\|f_j\|_{L^p}\lesssim 2^{-js}
  \eta_j$, with $s>0$ and $(\eta_j)_j\in l^q$. Then $g=\sum_j f_j\in
  \dot B^{s,q}_p$.
\end{lemma}
We have, by support conditions,
$$
g= \sum_k \Delta_k \sum_{k<j} S_j f_j.
$$
Now,
$$
\|\Delta_k (\sum_{k<j} S_j f_j)\|_p\lesssim 2^{-k
  s}\sum_{k<j}2^{-s(j-k)} \eta_j,
$$
which by an $l^1-l^q$ convolution provides the result.
\begin{lemma}
\label{paraLF}
  Let $f_j$ be such that $(I-S_j) f_j=f_j$, and $\|f_j\|_{L^p}\lesssim 2^{-js}
  \eta_j$, with $s<0$ and $(\eta_j)_j\in l^q$. Then $g=\sum_j f_j\in
  \dot B^{s,q}_p$.
\end{lemma}
We have, by support conditions,
$$
g= \sum_k \Delta_k \sum_{k>j} (I-S_j) f_j.
$$
Now,
$$
\|\Delta_k (\sum_{k>j} (I-S_j) f_j)\|_p\lesssim 2^{-k
  s}\sum_{k<j}2^{-s(j-k)} \eta_j,
$$
which by an $l^1-l^q$ convolution provides the result.
\begin{lemma}
\label{para}
  Consider $\alpha=1$ or $\alpha\geq 2$, $f\in \dot B^{s,q}_p$ and
  $g\in L^r$, with $0<s< 2$, $\frac 1
  m=\frac \alpha r +\frac 1 p$: let
$$
T^\alpha_g f=\sum_j (S_j g)^\alpha \Delta_j f.
$$
Then
$$
T^\alpha_g f\in \dot B^{s,q}_{m}.
$$
\end{lemma}
We split the \og paraproduct\fg{} $T^\alpha_g f$:
$$
T^\alpha_g f= \sum_j S_j ((S_j g)^\alpha \Delta_j f)+ \sum_j
(I-S_j) ((S_j g)^\alpha \Delta_j f);
$$
the first part is easily dealt with by Lemma \ref{paraHF}. For the
second one, $K_g f$, taking once again
advantage of the spectral supports
$$
\Delta_k K_g f= \Delta_k  \sum_{j<k}
(I-S_j) ((S_j g)^\alpha \Delta_j f).
$$
Notice the situation is close to the one in Lemma \ref{paraLF}, but we
don't have a negative regularity for summing. We therefore derive
\begin{eqnarray*}
  \Delta_D K_g f   & = & \sum_{j<k}(I-S_j) \Delta_D((S_j g)^\alpha \Delta_j f )\\
 & = &\sum_{j<k}(I-S_j) \left(\Delta_D( S_j g)^\alpha \Delta_j f +(\Delta_D
   \Delta_j f) (S_j g)^\alpha+2\alpha  (S_j g)^{\alpha-1}\nabla S_j g
   \cdot \nabla \Delta_j f \right)\\
 & = &\sum_{j<k}(I-S_j) \left(\alpha \Delta_D S_j g (S_j g)^{\alpha-1}
   \Delta_j f +\alpha(\alpha-1) |\nabla S_j g|^2 (S_j g)^{\alpha-2}
   \Delta_j f \right. \\
 & & {}+\left.(\Delta_D
  \Delta_j f) (S_j g)^\alpha+2\alpha (S_j g)^{\alpha-1} \nabla S_j g\cdot \nabla \Delta_j f \right).
\end{eqnarray*}
The first two pieces are again easily dealt with with Lemma
\ref{paraLF}, and the resulting function is in $\dot B^{s-2,q}_m$. The
remaining cross term is handled with some help from \cite{ip08}:
$$
\nabla \Delta_j f= \nabla \exp(4^{-j}\Delta_D) \tilde \Delta_j f,
$$ 
where the new dyadic block $\tilde \Delta_j$ is built on the function
$\tilde \psi(\xi)=\exp(|\xi|^2)\psi(\xi)$. From the continuity
properties of $\sqrt s \nabla \exp(s \Delta_D)$ on $L^p$,
$1<p<+\infty$, we immediatly deduce
\begin{equation}
  \label{eq:gradheat}
  \|\nabla \Delta_j f\|_p\lesssim 2^j \| \tilde \Delta_j f\|_p,
\end{equation}
and we can easily sum and
conclude. This will be enough to deal with the critical case, but for
differences of nonlinear power-like mappings, we need
\begin{lemma}
\label{diffpara}
  Consider $\alpha\geq 3$, $f,g\in X=\dot B^{s,q}_p\cap L^r$, with
  $0<s< 2$, $\frac 1
  m=\frac {\alpha-1} r +\frac 1 p$: Then, if $F(x)=|x|^{\alpha-1} x$ or $F(x)=|x|^{\alpha}$,
$$
\|F(u)-F(v)\|_{\dot B^{s,q}_{m}}\lesssim \|u-v\|_X (\|u\|^{\alpha-1}_X+\|v\|^{\alpha-1}_X).
$$
\end{lemma}
In order to obtain a factor $u-v$, we write
\begin{equation}
  \label{eq:diffF}
F(u)-F(v)=(u-v)\int_0^1 F'(\theta u+(1-\theta)v) d\theta.  
\end{equation}
We need to efficiently split this difference into two paraproducts
involving $u-v$ and $F'(w)$ with $w=\theta u+(1-\theta)v$, and this
requires an
estimate on $F'(w)$: write another telescopic series
\begin{align*}
F'(w)  = &\sum_j F'(S_{j+1} w)-F'(S_j w)  \\
       = &\sum_j S_j( F'(S_{j+1} w)-F'(S_j w))+\sum_j (I-S_j) (
       F'(S_{j+1} w)-F'(S_j w))\\
       = & S_1+S_2.
\end{align*}
Exactly as before, the first sum $S_1$ is easily disposed of with Lemma
\ref{paraHF}, as
$$
|F'(S_{j+1} w)-F'(S_{j} w)|\lesssim |\Delta_j w| (|S_{j+1}
w|^{\alpha-2}+|S_{j} w|^{\alpha-2}).
$$
The second sum $S_2$ requires again a trick; to avoid
uncessary cluttering, we set $F(x)=x^\alpha$, ignoring the sign issue
(recall that $\alpha\geq 3$, hence $F'''(x)$ is well-defined as a function): we apply
$\Delta_D$, let $\beta=\alpha-1\geq 2$
\begin{align*}
\Delta_D S_2 = & \sum_j (I-S_j)\Delta_D (
       (S_{j+1} w)^{\alpha-1}-(S_j w)^{\alpha-1})\\
       = & \sum_j (I-S_j) \left( \beta (S_{j+1} w)^{\beta-1}\Delta_D
         S_{j+1} w- \beta (S_{j} w)^{\beta-1}\Delta_D S_{j}
         w\right.\\
  & {}+\left. \beta(\beta-1) (S_{j+1} w)^{\beta-2}(\nabla S_{j+1} w)^2- \beta(\beta-1) (S_{j} w)^{\beta-2}(\nabla S_{j} w)^2 \right).
\end{align*}
We now apply Lemma \ref{paraLF} after inserting the right factors: we
have four types of differences,
\begin{align*}
  |((S_{j+1} w)^{\beta-1}-(S_{j}
  w)^{\beta-1})\Delta_D S_{j+1} w| & \lesssim C_\beta |\Delta_j w|
  |\Delta_D S_{j+1}| (|S_{j+1} w|^{\beta -2}+|S_{j} w|^{\beta -2})\\
  |(S_{j+1} w)^{\beta-1}\Delta_D \Delta_{j} w| & \leq |\Delta_D \Delta_j w|
  |S_{j+1} w|^{\beta -2}\\
|((S_{j+1} w)^{\beta-2}-(S_{j}
  w)^{\beta-2})(\nabla S_{j+1} w)^2| & \lesssim \tilde C_\beta |\Delta_j
  w|^{\beta -2}
  |\nabla S_{j+1} w|^{2}\\
|(S_{j+1} w)^{\beta-2}((\nabla S_{j}
  w)^{2}-(\nabla S_{j+1} w)^2)| & \leq |\nabla \Delta_j  w|(|\nabla
  S_j w|+|\nabla S_{j+1} w||S_{j+1} w|^{\beta -2}
  \end{align*}
where on the third line we wrote the worst case, namely $2\leq
\beta<3$ (otherwise the power of $\Delta_j w$ in the third bound will
be replaced by $|\Delta_j w|(|S_j w|^{\beta -3}+|S_{j+1} w|^{\beta -3})$).

By integrating, applying H\"older and using
\eqref{eq:gradheat} to eliminate the $\nabla$ operator, we obtain as
an intermediary result
$$
F'(w)\in \dot B^{s,q}_\lambda, \,\,\text{ with } \frac 1
\lambda=\frac{\alpha-2} r+\frac 1 p.
$$
We may now go back to the difference $F(u)-F(v)$ as expressed in
\eqref{eq:diffF} and perform a simple paraproduct decomposition in two
terms to which Lemma \ref{para} may be applied. Observe that there is
no difficulty in estimating $F'(w)$ in $L^{m/(\alpha-1)}$, and that
the integration in $\theta$ is irrelevant. This completes the proof.

We now go back to the first nonlinear estimate, namely \eqref{eq:map5}. We
write a telescopic series for the product five factors
$u_1,u_2,u_3,u_4,u_5 \in X_T$,
$$
u_1 u_2 u_3 u_4 u_5 =\sum_j S_{j+1}u_1 S_{j+1}u_2 S_{j+1}u_3
S_{j+1}u_4 S_{j+1}u_5 - S_{j} u_1 S_j u_2 S_j u_3 S_j u_4 S_j u_5
$$
and we are reduced to studying five sums of the same type, of which
the following is generic
$$
S_1=\sum_j \Delta_j u_1 S_j u_2 S_j u_3 S_j u_4 S_j u_5,
$$
and we intend to apply Lemma \ref{para}, which is trivially extended
to a product of several factors. In principle,
$$
u_k \in \BLT 1 5 2 {\frac{20}{11}} \cap L^{\frac {20} 3}_x L^{40}_T
$$
is enough, using the first space of the $\Delta_j$ factor and the
second one for all remaining $S_j$ factors, except for the use of
\eqref{eq:gradheat} in the proof. Consider, from $u\in X_T$,
$$
2^{\frac {11}{10} j} \|\Delta_j  u\|_{L^5_x L^2_T}+2^{-\frac {3}{2} j} \|\partial_t
\Delta_j  u\|_{L^5_T L^5_x}= \mu^{0}_j \in l^2_j.
$$
We will have, using \cite{ip08},
$$
2^{\frac {11}{10} j} \|\nabla \Delta_j  u\|_{L^5_x L^2_T}+2^{-\frac {3}{2} j} \|\partial_t
\nabla \Delta_j  u\|_{L^5_T L^5_x}= \mu^1_j \in l^2_j, \text{ with }
\|\mu^1\|_{l^2}\lesssim \|\mu^0\|_{l^2}.
$$
By Gagliardo-Nirenberg in time, we have the correct estimate for
$\Delta_j u$, for $k=0,1$
$$
2^{(1-k)j} \| \nabla^k \Delta_j u \|_{L^5_x L^{\frac{20}{11}}_T}\lesssim \mu^k_j.
$$
We proceed with the low frequencies by proving a suitable Sobolev
embedding.
\begin{lemma}
\label{Sob}
  Let $u\in \BLT {\frac 1 2} 5 5 5$ and $\partial_t u \in \BLT {-\frac 3
    2} 5 5 5$. Then $u \in L^{\frac {20} 3}_x L^{40}_T$.
\end{lemma}
Let
$$
2^{(\frac {1}{2}-k) j} \|\nabla^k\Delta_j  u\|_{L^5_x L^5_T}+2^{-(k+\frac {3}{2}) j} \|\partial_t\nabla^k
\Delta_j  u\|_{L^5_T L^5_x}= \mu^{k}_j \in l^5_j,
$$
notice we can easily switch time and space Lebesgue norms. Using 
Gagliardo-Nirenberg in time, we have
\begin{equation}
  \label{eq:up}
  2^{(\frac {1}{6}-k) j} \|\nabla^k \Delta_j  u\|_{L^5_x L^{30}_T}\lesssim \mu^{3}_j \in l^5_j.
\end{equation}
Using now Gagliardo-Nirenberg in space, we also have
$$
2^{-\frac j {10}}\|\Delta_j u\|_{ L^\infty_x L^5_T}\lesssim 2^{-\frac
  j {10}}\|\Delta_j u\|_{ L^5_T L^\infty_x} \lesssim \mu^5_j
$$
and the same thing for $2^{-2j}\partial_t \Delta_j u$ (or with an
additional $2^j \nabla$). Now another
Gagliardo-Nirenberg in time provides
\begin{equation}
  \label{eq:down}
  2^{-(k+\frac 1 2)j} \|\nabla^k \Delta_j u\|_{L^\infty_{T,x}} \lesssim \mu^6_j.
\end{equation}
Finally, we take advantage of a discrete embedding between
$l^1$ and weighted $l^\infty$ sequences:
\begin{align*}
  | u |  & \leq \sum_{j<J} |\Delta_j u|+\sum_{j\geq J} |\Delta_j u|\\
         & \leq \sum_{j<J} 2^{\frac j 2} \sup_{j} 2^{-\frac j
           2}|\Delta_j u|+\sum_{j\geq J} 2^{-\frac j 6} \sup_{j}
         2^{\frac j 6} |\Delta_j u|\\
          & \lesssim  2^{\frac J 2} \sup_{j} 2^{-\frac j
           2}|\Delta_j u|+ 2^{-\frac J 6} \sup_{j}
         2^{\frac j 6} |\Delta_j u|\\
|u|^4  & \lesssim   \sup_{j} 2^{-\frac j
           2}|\Delta_j u| \left(\sup_{j}
          2^{\frac j 6} |\Delta_j u|\right)^3\\
\||u|^4\|_{L^\frac 5 3_x L^{10}_T}  & \lesssim  \| \sup_{j} 2^{-\frac j
           2}|\Delta_j u|\|_{L^\infty_{T,x}}\| \sup_{j}
          2^{\frac j 6} |\Delta_j u|\|^3_{L^5_x L^{30}_T}\\
\|u\|_{L^\frac {20} 3_x L^{40}_T}  & \lesssim  \| u|\|_{\BL {\frac 1
    2} \infty \infty \infty}^{\frac 1 4} \| u|\|^\frac 3 4_{\BL {\frac
    1 6} 5 5 {30}}\\
\end{align*}
Notice that the estimate with a gradient is much easier: just
interpolate between \eqref{eq:up} and \eqref{eq:down} with $k=1$ to
obtain
$$
2^{-j} \|\nabla \Delta_j u\|_{L^\frac {20} 3_x L^{40}_T}\lesssim
\mu^7_j,
$$
which we can now sum over $k<j$ to obtain control of $S_j u$.

The case $p<5$ is handled in an similar way, and we leave the details
to the reader, sparing him the complete set of exponents (depending on
$p$ !) that would appear in the proof. For scaling reasons there is
actually no need to perform the computation: the previous one on the
critical case simply illustrates that we can sidestep issues related
to the usual Littlewood-Paley theory by using direct arguments.

\bibliography{nls_energy_critical}

\begin{thebibliography}{10}

\bibitem{ramona}
Ramona Anton.
\newblock Global existence for defocusing cubic {NLS} and {G}ross-{P}itaevskii
  equations in three dimensional exterior domains.
\newblock {\em J. Math. Pures Appl. (9)}, 89(4):335--354, 2008.

\bibitem{bgt03}
N.~Burq, P.~G{\'e}rard, and N.~Tzvetkov.
\newblock On nonlinear {S}chr\"odinger equations in exterior domains.
\newblock {\em Ann. Inst. H. Poincar\'e Anal. Non Lin\'eaire}, 21(3):295--318,
  2004.

\bibitem{bu02}
Nicolas Burq.
\newblock Estimations de {S}trichartz pour des perturbations à longue portée de
  l'opérateur de {S}chr\"odinger.
\newblock In {\em S\'eminaire: \'{E}quations aux {D}\'eriv\'ees {P}artielles.
  2001--2002}, S\'emin. \'Equ. D\'eriv. Partielles. \'Ecole Polytech.,
  Palaiseau, 2002.

\bibitem{bulepl07}
Nicolas Burq, Gilles Lebeau, and Fabrice Planchon.
\newblock Global existence for energy critical waves in 3-{D} domains.
\newblock {\em J. Amer. Math. Soc.}, 21(3):831--845, 2008.

\bibitem{buplck}
Nicolas Burq and Fabrice Planchon.
\newblock Smoothing and dispersive estimates for 1{D} {S}chr\"odinger equations
  with {BV} coefficients and applications.
\newblock {\em J. Funct. Anal.}, 236(1):265--298, 2006.

\bibitem{bupl07}
Nicolas Burq and Fabrice Planchon.
\newblock Global existence for energy critical waves in 3-d domains : Neumann
  boundary conditions, 2007.
\newblock to appear in Amer. J. of Math., {\tt arXiv:math/0711.0275}.

\bibitem{cawe90}
Thierry Cazenave and Fred~B. Weissler.
\newblock The {C}auchy problem for the critical nonlinear {S}chr\"odinger
  equation in {$H\sp s$}.
\newblock {\em Nonlinear Anal.}, 14(10):807--836, 1990.

\bibitem{chki01}
Michael Christ and Alexander Kiselev.
\newblock Maximal functions associated to filtrations.
\newblock {\em J. Funct. Anal.}, 179(2):409--425, 2001.

\bibitem{ckstt06}
J.~Colliander, M.~Keel, G.~Staffilani, H.~Takaoka, and T.~Tao.
\newblock Global well-posedness and scattering for the energy-critical
  nonlinear {S}chr\"odinger equation in {$\mathbb{R}\sp 3$}.
\newblock {\em Ann. of Math. (2)}, 167(3):767--865, 2008.

\bibitem{give85}
J.~Ginibre and G.~Velo.
\newblock The global {C}auchy problem for the nonlinear {S}chr\"odinger
  equation revisited.
\newblock {\em Ann. Inst. H. Poincar\'e Anal. Non Lin\'eaire}, 2(4):309--327,
  1985.

\bibitem{OanaPS}
Oana Ivanovici.
\newblock Precise smoothing effect in the exterior of balls.
\newblock {\em Asymptot. Anal.}, 53(4):189--208, 2007.

\bibitem{OanaCE}
Oana Ivanovici.
\newblock Counter example to {S}trichartz estimates for the wave equation in
  domains, 2008.
\newblock to appear in Math. Annalen, {\tt arXiv:math/0805.2901}.

\bibitem{OanaSS}
Oana Ivanovici.
\newblock On the {S}chrodinger equation outside strictly convex obstacles,
  2008.
\newblock {\tt arXiv:math/0809.1060}.

\bibitem{ip08}
Oana Ivanovici and Fabrice Planchon.
\newblock Square function and heat flow estimates on domains, 2008.
\newblock {\tt arXiv:math/0812.2733}.

\bibitem{pl02}
Fabrice Planchon.
\newblock Dispersive estimates and the 2{D} cubic {NLS} equation.
\newblock {\em J. Anal. Math.}, 86:319--334, 2002.

\bibitem{plve08}
Fabrice Planchon and Luis Vega.
\newblock Bilinear virial identities and applications.
\newblock {\em Ann. Scient. Éc. Norm. Sup.}, 42:261--290, 2009.

\bibitem{smso95}
Hart~F. Smith and Christopher~D. Sogge.
\newblock On the critical semilinear wave equation outside convex obstacles.
\newblock {\em J. Amer. Math. Soc.}, 8(4):879--916, 1995.

\bibitem{smso06}
Hart~F. Smith and Christopher~D. Sogge.
\newblock On the {$L\sp p$} norm of spectral clusters for compact manifolds
  with boundary.
\newblock {\em Acta Math.}, 198(1):107--153, 2007.

\bibitem{stta02}
Gigliola Staffilani and Daniel Tataru.
\newblock Strichartz estimates for a {S}chr\"odinger operator with nonsmooth
  coefficients.
\newblock {\em Comm. Partial Differential Equations}, 27(7-8):1337--1372, 2002.

\end{thebibliography}
\end{document}